\documentclass[final]{siamltex}
\usepackage{a4wide, amstext, amsmath, amssymb, graphicx, array, epsfig, psfrag}
\usepackage{amstext, amsmath, amssymb, graphicx}

\usepackage{ifdraft}
\ifoptionfinal{
\usepackage[disable]{todonotes}
}{
\usepackage[norefs, nocites]{refcheck}

\usepackage[notref, notcite]{showkeys}
\usepackage[bordercolor=white, color=white]{todonotes}
}

\makeatletter\providecommand\@dotsep{5}\def\listtodoname{List of Todos}\def\listoftodos{\hypersetup{linkcolor=black}\@starttoc{tdo}\listtodoname\hypersetup{linkcolor=blue}}\makeatother

\newcommand{\bfn}{\boldsymbol n}

\newcommand{\bfR}{\boldsymbol R}

\newcommand{\bfx}{\boldsymbol x}

\newcommand{\bfb}{\boldsymbol b}
\newcommand{\bfp}{\boldsymbol p}
\newcommand{\bfq}{\boldsymbol q}
\newcommand{\bfy}{\boldsymbol y}

\newcommand{\bfF}{\boldsymbol F}
\newcommand{\bfxi}{\boldsymbol \xi}
\newcommand{\bfnu}{\boldsymbol \nu}
\newcommand{\bfeta}{\boldsymbol \eta}

\newcommand{\tn}{|\mspace{-1mu}|\mspace{-1mu}|}

\numberwithin{equation}{section}

\newtheorem{prop}{Proposition}[section]
\newtheorem{thm}{Theorem}[section]
\newtheorem{rem}{Remark}[section]
\newtheorem{cor}{Corollary}[section]
\newcommand{\jump}[1]{[\![#1]\!]}
\newcommand{\tnorm}[1]{\tn#1\tn}

\newcommand{\hK}{h}
\newcommand{\piF}{\pi_{F,l}}

\DeclareMathOperator*{\argmin}{arg\,min}
\title{
{Primal dual mixed finite element
  methods for the elliptic Cauchy problem}}

\author{Erik Burman\thanks{Department of Mathematics, 
University College London, Gower Street, London, 
UK--WC1E  6BT, 
United Kingdom; ({\tt e.burman@ucl.ac.uk})}
\and Mats G. Larson \thanks{Department of Mathematics and Mathematical Statistics, Ume{\aa} University, 
SE-901 87 Ume{\aa}, Sweden; ({\tt mats.larson@math.umu.se})}
\and Lauri Oksanen
\thanks{Department of Mathematics, University College London, 
Gower Street, London UK, WC1E 6BT, ({\tt l.oksanen@ucl.ac.uk})}
}

\begin{document}
\maketitle
\numberwithin{equation}{section} \maketitle
\begin{abstract}
We consider primal-dual mixed finite element methods for the solution
of the elliptic Cauchy problem, or other related data assimilation
problems. The method has a local conservation property. We derive a priori error estimates using known conditional stability
estimates and determine the minimal amount of weakly consistent
stabilization and Tikhonov regularization that yields optimal convergence for smooth exact solutions. The effect of perturbations in data
is also accounted for.  A reduced version of the method, obtained by choosing a
special stabilization of the dual variable, can be viewed as a variant of the 
least squares mixed finite element method introduced by Dard\'e, Hannukainen 
and Hyv\"onen in \emph{An {$H\sb {\sf{div}}$}-based mixed quasi-reversibility method
for solving elliptic {C}auchy problems}, SIAM J. Numer. Anal., 51(4) 2013. 
The main difference is that our choice of regularization does not depend on auxiliary 
parameters, the mesh size being the only asymptotic parameter.
Finally, we show that the reduced method can be used for defect correction 
iteration to determine the solution of the full method. 
The theory is illustrated by some numerical examples.
\end{abstract}


\section{Introduction}
Let $\Omega \in \mathbb{R}^d$, $d \in \{2,3\}$, be a convex polygonal/polyhedral domain,
with boundary $\partial \Omega$ and outward pointing unit normal $\bfnu$. We consider the
following elliptic Cauchy problem,
\begin{equation}\label{Cauchy_strong}
\begin{array}{rcl}
\nabla \cdot  (A  \nabla u) + \mu u &=& f + \nabla \cdot \bfF \mbox{ in } \Omega \\
u & = & g \mbox{ on } \Sigma \\
(A \nabla u )\cdot \bfnu & = & \psi \mbox{ on } \Sigma,
\end{array}
\end{equation}
where $\Sigma \subset \partial \Omega$. 
The problem data is given by $f \in
L^2(\Omega)$, $\bfF \in [L^2(\Omega)]^d$, $g \in
H^{\frac12}(\Sigma)$, $\psi \in
H^{-\frac12}(\Sigma)$, $\mu \in \mathbb{R}$. The diffusivity matrix $A
\in \mathbb{R}^{d\times d}$ is assumed to be symmetric positive
definite. Observe that the
second term in the right-hand side is well defined only in the weak sense, see (\ref{weak_Cauchy}) below for the precise formulation.
For the physical problem, the function $\bfF$ will be assumed to be zero, but it
will play a role for the numerical analysis. 

Contrary to a typical boundary value problem, the data $g$, $\psi$ is available only on the portion $\Sigma$ of the domain
boundary. 
Observe that on this portion, on the other hand, both the Dirichlet and
the Neumann data are known. For simplicity we only consider the case of
unperturbed Dirichlet boundary conditions.
We will assume that $\psi$ is some measured Neumann data, possibly
with perturbations $\delta \psi$. We also assume that the unperturbed
data has at least the additional regularity $g \in H^{\frac32}(\Sigma)$ and
$\psi \in H^{\frac12}(\Sigma)$ and that there
exists a solution $u \in
H^2(\Omega)$ to \eqref{Cauchy_strong} for the given $f, \, \psi$ and $g$. The elliptic Cauchy problem is
severly ill-posed \cite{Belg07} and even when a unique solution $u$
exists, small perturbations of data in the computational model can
have a strong impact on the result.

The computational approximation of ill-posed problems is a challenging topic. Indeed the
lack of stability of the physical model under study typically prompts Tikhonov
regularization on the continuous level \cite{LL69,TA77} in order to obtain a well-posed
problem, which then allows for standard approximation techniques to be
applied. Although convenient, this approach comes with the price of
having to estimate both the perturbation error induced by adding the
regularization and the approximation error due to discretization, in
order to assess the quality of the solution. For early works on finite element 
approximation of the elliptic Cauchy problem and ill-posed problems we
refer to \cite{Han82,FM86, Falk90, RHD99}.

Herein we will advocate a different
approach based on discretization of the ill-posed physical model in an
optimization framework,
followed by regularization of the discrete problem. This primal-dual approach was
first introduced by Burman in the papers \cite{Bu13, Bu14a,Bu14b,Bu16}, drawing
on previous work by Bourgeois and Dard\'e  on quasi reversibility
methods \cite{Bour05,Bour06,BD10a,BD10b} and further
developed for elliptic data assimilation problems \cite{BHL16},  for
parabolic data
reconstruction problems in \cite{BO16, BIHO17} and finally for unique
continuation for Helmholtz equation \cite{BNO17}. For a related method
using finite element spaces with $C^1$-regularity see \cite{CM15} and
for methods 
designed for well-posed, but indefinite problems, we refer to \cite{BLP97}
and for second order elliptic problems on non-divergence form see
\cite{WW15} and \cite{WW17}. Recently approaches similar to those
discussed in this work were proposed for the approximation of
well-posed convection--diffusion problems \cite{FK17} or porous media
flows \cite{LWZ17}.

The idea is to cast the ill-posed problem on
the form of an optimization problem under the constraint of the
satisfaction of the partial differential equation, and look for the solution of the
discrete form of the partial differential equation that allows for the best 
matching of the data. This problem is unstable also on the discrete level and to improve the
stability properties we use stabilization techniques known from the
theory of stabilized finite element methods. Typical stabilizers are least
squares penalty terms on fluctuations of discrete quantities over
element faces, or Galerkin least squares term on the residual, in the elements. Since both a forward and a dual problem must be solved, this approach 
doubles the number of degrees of freedom in the computation. 

The objective of the present work is to revisit the primal-dual
stabilized method for the Cauchy problem but in the context of mixed
finite element methods. This means that we use one variable to
discretize the flux variable and another for the primal variable. In
this framwork, the primal stabilizer, that typically is based on the penalization of fluctuations, can be formulated as the
difference between the flux variable and the flux evaluated using the
primal variable. Our method is designed by minimizing this fluctuation
quantity under the constraint of the conservation law. The use of the
mixed finite element formulation allows us to choose discrete spaces
in such a way that the conservation law is satisfied exactly on each
cell of the mesh. The resulting
system is large, but we show that a special choice of the
adjoint stabilizer allows for the elimination of the multiplier and a
reduction of the system to a symmetric least squares formulation, at
the price of exact local conservation. For the reduced case
local conservation is only satisfied asymptotically. 

The reduced method is identified as a variant of the method proposed by 
Dard\'e, Hannukainen and Hyv{\"o}nen in \cite{DHH13}. In this work the 
elliptic Cauchy problem was considered and a discrete solution was sought 
using Raviart-Thomas finite elements for the flux variable and standard 
Lagrange elements for the primal variable. Additional stability was obtained 
through Tikhonov regularization on both the primal and the flux variable. 
Contrary to \cite{DHH13}, our choice of regularization does not depend 
on auxiliary parameters, the mesh size being the only asymptotic parameter. 
This allows us to carry out a complete analysis of the rate of convergence of the method.

The convergence analysis is based on known conditional stability estimates for the elliptic
Cauchy problem, see e.g. \cite{ARRV09}. We prove estimates for the mixed FEMs that
in a certain sense can be considered optimal with respect to the
approximation order of the space, the stability properties of the
ill-posed problem and perturbations in data. For the
analysis using conditional stability estimates, we need an a priori bound on the 
discrete solution. This naturally leads to the introduction of a Tikhonov regularization on
the primal variable. The dependence of the regularization parameter on
the mesh-size is chosen so that optimal convergence is obtained for
unperturbed solutions, depending on the approximation order of the
space and the regularity of the exact solution. The analysis is
illustrated with some numerical examples. 

\subsection{The elliptic Cauchy problem}
The problem \eqref{Cauchy_strong} can be cast on weak form by
introducing the spaces
\[
V_{\Sigma}:= \{ v \in H^1(\Omega) : v\vert_{\Sigma} = g \}
\]
and, using $\Sigma' := \partial \Omega \setminus \Sigma$,
\[
V_{\Sigma'}:= \{ v \in H^1(\Omega) : v\vert_{\Sigma'} = 0 \}.
\]
We also introduce the bilinear forms
\[
a(u,v) := \int_{\Omega} (-A \nabla u \cdot \nabla v  + \mu u v )~
\mbox{d}x, \quad l(v) := \int_{\Sigma} \psi v ~
\mbox{d}s + \int_{\Omega} f v ~
\mbox{d}x + \int_{\Omega} \bfF \cdot \nabla v ~
\mbox{d}x.
\]
The weak formulation then reads: find $u \in V_{\Sigma}$ such that 
\begin{equation}\label{weak_Cauchy}
a(u,v) = l(v) \quad \forall v \in V_{\Sigma'}.
\end{equation}

Observe that this problem is severely ill posed (see e.g.
\cite{Belg07}). 
Moreover, if $f$ and $\psi$ are chosen arbitrarily, it may fail to have a solution. 
We will assume below that 
we have at our disposal perturbed data, 
$$\tilde \psi := \psi + \delta \psi, \quad \delta \psi \in L^2(\Sigma)
$$ 
and 
$$
\tilde f = f + \delta f, \quad \delta f \in L^2(\Omega),
$$ 
such that, in the unperturbed case $(\delta \psi = 0,\, \delta f = 0)$, there 
is a solution $u \in H^2(\Omega)$ of (\ref{weak_Cauchy}). We then arrive 
to the following perturbed problem, find $u \in V_\Sigma$ such that
\[
a(u,v) = \tilde l(v),\quad \forall v \in V_{\Sigma'},
\]
where the perturbed right hand side is given by
\[
\tilde l(v) := \int_{\Sigma} \tilde \psi v ~
\mbox{d}s + \int_{\Omega} \tilde f v ~
\mbox{d}x.
\]
Here we have omitted the contribution from $\bfF$ since this term is
assumed to be zero for the physical problem. The problem posed with the
perturbed data most likely does not have a solution.

We define for $k \ge 0$,
\[
H^k_{div,\psi} := \{\bfq \in H^k(\Omega): \nabla \cdot \bfq \in
H^k(\Omega) \mbox{ and } \bfq \cdot \bfnu\vert_\Sigma = \psi\}.
\]
Assuming $f \in L^2(\Omega)$, the flux variable $\bfp := A \nabla
u$ is in $H_{div,\psi} := H^0_{div,\psi}$.
We will write also $H_{div} = H_{div,0}$. For the finite element method we
will use the equation \eqref{Cauchy_strong} written on mixed form,
that is, find $u \in V_\Sigma, \, \bfp \in H_{div,\psi}$ such that
\begin{align}\label{eq:mixed_form1}
\bfp - A \nabla u& =
0 \mbox{ in } \Omega,\\
\label{eq:mixed_form2}
\nabla \cdot \bfp + \mu u  & = f \mbox{ in } \Omega.
\end{align}
The method that we will propose below will be based on minimizing the
left hand side of \eqref{eq:mixed_form1} under the constraint of
\eqref{eq:mixed_form2}.

In the analysis below we will use the following notation for the $L^2$-scalar products and
norms on
$\omega \subset \mathbb{R}^d$ and $\sigma \subset \mathbb{R}^{d-1}$,
\[
(u,v)_\omega:=\int_\omega u v ~\mbox{d}x, \mbox{ with norm }
\|u\|_\omega = (u,u)_\omega^{\frac12}
\]
and
\[
\left< u,v \right>_\sigma:=\int_\sigma u v ~\mbox{d}s, \mbox{ with norm }
\|u\|_\sigma =\left< u,u \right>_\sigma^{\frac12}.
\]
With some abuse of notation we will not distinguish between the norms
of vector valued and scalar quantities.
\subsection{Stability properties of the Cauchy problem}
The literature on the stability properties of
the elliptic Cauchy problem spans more than a hundred years, see for instance \cite{Hada02,Pay60,Pay70,ABF06,Belg07,ARRV09,BD10c}. 
The results known as quantitative unique
continuation or quantitative uniqueness are useful for numerical
analysis. For our analysis we will use the results in 
\cite{ARRV09},
and we refer the reader to this review paper by Allessandrini and 
his co-authors for background on the analysis of
the Cauchy problem. To keep down the technical detail we will here
present their main results on a simplified form suitable for our
analysis. In particular, we do not track the constants related to the
geometry of the domain. For the complete results, as
well as full proofs, we refer to \cite{ARRV09}.
First we introduce the following bounds on the data. Assume that there
exists $\eta, \varepsilon >0$ such that
\begin{equation}\label{bc_data}
\|g\|_{H^{1/2}(\Sigma)} + \|\psi\|_{H^{-1/2}(\Sigma)} \leq \eta
\end{equation}
and $\|f\|_{\Omega} + \|\bfF\|_{\Omega} \leq \varepsilon$ or
equivalently, the right hand side $l(v)$ satifies the following bound
\begin{equation}\label{bulk_data}
\|l\|_{(V_{\Sigma'})'} \leq \varepsilon
\end{equation}
\begin{thm}\label{locbound}(Conditional stability of the Cauchy problem, local bound)
Assume that $u \in H^1(\Omega)$ is a solution to
\eqref{weak_Cauchy}, with data satisfying \eqref{bc_data} and \eqref{bulk_data}. Assume that the following a priori bound holds
\begin{equation}\label{eq:aprioriloc}
\|u\|_{\Omega} \leq E_0.
\end{equation}
Let $G \subset \Omega$ be such that $\mbox{dist}(G,\Sigma') > 0$. Then there exists a constant
$C>0$ and $\tau \in (0,1)$ depending only on the geometry of
$\Omega$ and $G$ such that
\begin{equation}\label{eq:locbound}
\|u\|_{G} \leq C (\varepsilon + \eta)^{\tau} (E_0 + \varepsilon + \eta)^{(1-\tau)}.
\end{equation}
\end{thm}

Observe that compared to Theorem 1.7 in \cite{ARRV09},
we have omitted the assumption that $\mbox{dist}(G,\Sigma)$
is small. This is because estimates for $u$ can be propagated in the interior of $\Omega$, at the cost of making the constants $C$ and $\tau$ worse, see Section 5 in \cite{ARRV09}.
It is important, however, that $G$ does not touch $\Sigma'$. If it does, the optimal estimate is of logarithmic type.

\begin{thm}\label{globbound}(Conditional stability of the Cauchy problem, global bound)
Assume that $u \in H^1(\Omega)$ is a solution to
\eqref{weak_Cauchy}, with data satisfying \eqref{bc_data} and \eqref{bulk_data}. Assume that the following a priori bound holds
\begin{equation}\label{eq:aprioriglob}
\|u\|_{H^1(\Omega)} \leq E.
\end{equation}
Then there exists a constant
$C>0$ and $\tau \in (0,1)$ depending only on the geometry of
$\Omega$ such that
\begin{equation}\label{eq:globbound}
\|u\|_{\Omega} \leq C (E + \varepsilon + \eta) \mathop{\omega}\left(\frac{\varepsilon + \eta}{E + \varepsilon + \eta}\right)
\end{equation}
where
\[
\omega(t) \leq \frac{1}{\log(t^{-1})^\tau},\quad \mbox{ for } t<1.
\]
\end{thm}

\section{The mixed finite element framework}
Let $\{\mathcal{T}\}_h$ be a family of conforming, quasi uniform meshes
consisting of shape regular simplices $\mathcal{T} = \{K\}$. The index
$h$ is the mesh parameter $h$, defined as the largest diameter of any
element $K$ in $\mathcal{T}$. For each
simplex $K$ we  let $\bfn_K$ be the outward pointing unit normal. We 
assume that the boundary faces of $\mathcal{T}$ fits the zone $\Sigma$ 
so that $\partial \Sigma$ nowhere cuts through a boundary face. The
set of faces of the elements in $\mathcal{T}$ will be denoted by $\mathcal{F}$
and the set of faces in $\mathcal{F}$ whose union coincides with $\Sigma$ 
by $\mathcal{F}_\Sigma$.

We introduce the space of functions in $L^2(\Omega)$ that are piecewise 
polynomial of order $k$ on each element 
\[
X^k_h := \{x_h \in L^2(\Omega): x_h\vert_K \in \mathbb{P}_k(K), \, \forall K \in \mathcal{T}\},
\]
where $\mathbb{P}_k(K)$ denotes the set of polynomials of degree less
than or equal to $k$ on the simplex $K$. We define the $L^2$-projection 
$\pi_{X,k}:L^2(\Omega) \mapsto X_h^k$ by, $\pi_{X,k} y \in X^k_h$ such that
\[
(\pi_{X,k}  y - y,v_h)_\Omega = 0,\quad \forall v_h \in X_h^k.
\]
The
$L^2$-projection on a face $F$ of some simplex $K \in \mathcal{T}$,
will also be used in the analysis. We define $\piF:L^2(F) \mapsto
\mathbb{P}_{l}(F)$ such that, for $\phi \in L^2(F)$, $\piF \phi$ satisfies
\[
\left<\phi- \piF \phi, p_h \right>_F = 0,\quad \forall p_h \in \mathbb{P}_l(F).
\]
For functions in $X^k_h$ we introduce the broken norms,
\begin{equation}\label{eq:discrete_H1}
\|x\|_h := \left(\sum_{K \in \mathcal{T}} \|x\|_K^2\right)^{\frac12}\mbox{
    and } \|x\|_{1,h} := \left(\|\nabla x\|^2_h +
  \|h^{-\frac12}\piF \jump{x}\|^2_{\mathcal{F} \setminus \mathcal{F}_\Sigma} \right)^{\frac12}
\end{equation}
where $\|x\|^2_{\mathcal{F}}:= \sum_{F
  \in \mathcal{F}} \|x\|_F^2$ and
\[
\jump{u}\vert_F(x) := \left\{ \begin{array}{l}
\lim_{\epsilon
  \rightarrow 0^+} (u(x - \epsilon \bfn_F) - u(x + \epsilon \bfn_F))
                             \mbox{ for } F \in \mathcal{F}_i\\
u(x) \mbox{ for } F \in \mathcal{F}_{\Sigma'}
\end{array}
\right. 
\]
where $\bfn_F$ is a fixed unit normal to the face $F$ and $\mathcal{F}_i$ is the set of 
interior faces. Note that we do not need to define the jump on $\Sigma$.
Also recall the discrete Poincar\'e inequality \cite{Brenner03},
\begin{equation*}
\|x\|_{L^2(\Omega)}\lesssim \|x\|_{1,h}, \quad \forall x \in X_h^k,
\end{equation*}
which guarantees that the right expression of \eqref{eq:discrete_H1} is a norm. Here and below 
we use the notation $a \lesssim b$ for $a \leq C b$ where $C$ is a 
constant independent of $h$.
Occasionally, we will also use the notation $a \sim b$ meaning $a \lesssim b$
and $b \lesssim a$. 

To formulate the method we write the standard $H^1$-conforming finite
element space 
\[
L_h^k:=\{v_h \in H^1(\Omega) \cap X^k_h \}.
\]
For the primal variable it is convenient to introduce the spaces
\[
V_{g}^k:=\{v_h \in L_h^k\, : \,v_h = g_h \mbox{ on } \Sigma \}, \quad V_{0}^k:=\{v_h \in L_h^k\, : \,v_h = 0 \mbox{ on } \Sigma \}.
\]
We let $g_h$ denote the nodal interpolant of $g$ on the trace of
functions in $V_h$ on
$\Sigma$, so that
defining the nodal interpolant $i_h:C^0(\bar \Omega) \mapsto
L_h^k$, there holds $i_h:V_\Sigma \mapsto V_{g}^k$. The following
approximation estimate is satisfied by $i_h$, see e.g. \cite{EG04}. For $v \in
H^{k+1}(\Omega)$ there holds
\begin{equation}\label{eq:approx_Lagrange}
\|v - i_h v\|_\Omega + h \|\nabla (v - i_h v)\|_\Omega  \lesssim
h^{k+1} |v|_{H^{k+1}(\Omega)}, \quad k\ge 1.
\end{equation}
The flux variable will be approximated in the Raviart-Thomas space
\[
RT^l:= \{\bfq_h \in H_{div}(\Omega)\, : \, \bfq_h\vert_K \in
\mathbb{P}_l(K)^d\oplus \bfx (\mathbb{P}_l(K) \setminus
\mathbb{P}_{l-1}(K))\mbox{ for all } K \in \mathcal{T} \}
\]
with $\bfx \in \mathbb{R}^d$ being the spatial variable. 
We recall the Raviart-Thomas interpolant
$\bfR_h:H_{div}(\Omega) \mapsto RT^l$ and its approximation
  properties \cite{EG04}.  For $\bfq \in
H^l_{div}(\Omega)$ and $\bfR_h \bfq \in RT^{l}$ there holds
\begin{equation}\label{eq:approx_RT}
\|\bfq  - \bfR_h \bfq\|_\Omega + \|\nabla \cdot (\bfq  - \bfR_h
\bfq)\|_\Omega \lesssim h^{l} (|\nabla \cdot \bfq |_{H^{l}(\Omega)}+|\bfq |_{H^{l}(\Omega)}).
\end{equation}

Then assuming that
the Neumann data $\tilde \psi$ is in $L^2(\Sigma)$ we define the discretized Neumann boundary data by the
$L^2$-projection, for $F \in \Sigma$, $\tilde \psi_h\vert_F := \piF \tilde \psi$.
A space for the flux variable, with the
satisfaction of the Neumann condition built in, takes the form
\[
\begin{array}{l}
D_{\tilde \psi}^l:=\{\bfq_h \in RT^{l} : \bfq_h \cdot \bfnu = \tilde
  \psi_h \mbox{ on } \Sigma \},\quad
D_{0}^l:=\{\bfq_h \in RT^{l} : \bfq_h \cdot \bfnu = 0 \mbox{ on } \Sigma \}.
\end{array}
\]

Given a function $x_h \in X_h^k$ we define a reconstruction  $\bfeta_h(x_h)$ 
of the gradient of $x_h$ in $D_{0}^l$.
By the properties of the Raviart-Thomas element there exists $\bfeta_h(x_h)
\in D_{0}^l$ such that for all $F \in \mathcal{F} \setminus
\mathcal{F}_\Sigma$ 
\begin{equation}\label{eq:reconstruct1}
\left<\bfeta_h(x_h) \cdot \bfn_F,  w_h \right>_F= \left< h_F^{-1} \jump{x_h},
w_h \right>_F, \mbox{ for all } w_h \in \mathbb{P}_{l}(F)
\end{equation}
where $h_F$ is the diameter of $F$, and if $l \ge 1$, for all $K \in \mathcal{T}$,
\begin{equation}\label{eq:reconstruct2}
( \bfeta_h(x_h) ,\bfq_h)_K = -(\nabla x_h,
\bfq_h )_K, \mbox{ for all } \bfq_h \in [\mathbb{P}_{l-1}(K)]^d.
\end{equation}
The stability of $\bfeta_h$ with respect to data is crucial in the
analysis below and we therefore prove it in a proposition.
\begin{proposition}
There exists a unique $\bfeta_h \in D_0^l$ such that
\eqref{eq:reconstruct1}-\eqref{eq:reconstruct2} hold for every face $F \in \mathcal{F} \setminus
\mathcal{F}_\Sigma$
and every element
in the mesh. More over $\bfeta_h$ satisfies the stability estimate
\begin{equation}\label{eq:etastab}
\|\bfeta_h\|_\Omega \leq C_{ds} (\|\pi_{X,l-1} \nabla x_h\|^2_{h} +
\|h^{-\frac12} \piF \jump{x_h}\|^2_{\mathcal{F}\setminus \mathcal{F}_\Sigma})^{\frac12},
\end{equation}
here $C_{ds}>0$ is a constant depending only on the element
shape regularity that will appear in the constant of the stability estimate, see Proposition \ref{infsup} below.
\end{proposition}
\begin{proof}
The unique existence of $\bfeta_h$ is an immediate consequence of the definition and unisolvence of the Raviart-Thomas space. Observe that the left hand side of equations \eqref{eq:reconstruct1}-\eqref{eq:reconstruct2}  coincides exactly with the degrees of freedom 
defining the Raviart-Thomas element.

For the stability estimate (\ref{eq:etastab}) we notice that since by definition
$\pi_{K,l-1} \bfeta_h\vert_K = \pi_{K,l-1} \nabla x_h\vert_K$ and
$\bfeta_h \cdot \bfn_K \vert_{\partial K} = h_F^{-1} \piF \jump{x_h}\vert_K$ it is 
enough to prove the estimate
\[
\|\bfeta_h\|^2_K \lesssim \|\pi_{K,l-1} \bfeta_h\|^2_K + \hK \|\bfeta_h \cdot \bfn_K\|^2_{\partial K}.
\]
To this end let $\hat K$ be a fixed reference element. Then we have the bound
\begin{equation}\label{eq:stab-ref-element}
\|\hat \bfeta_h\|^2_{\hat K} 
\lesssim 
\|\pi_{\hat K,l-1} \hat \bfeta_h \|^2_{\hat K} 
+ \| \hat \bfeta_h \cdot \hat \bfn_{\hat K} \|^2_{\partial \hat K}
\end{equation}
by finite dimensionality and unisolvence of the Raviart-Thomas element. 

Next let $\Phi(\hat \bfx) = \bfb + B \hat \bfx$ be an affine mapping such that 
$\Phi:\hat K \rightarrow K$ is a bijection and the determinant $|B|$ of $B$ is positive. 
Define the mappings $w = \hat w \circ \Phi^{-1}$ and  $\bfq = |B|^{-1} B \hat \bfq \circ \Phi^{-1}$. 
Then we have identities $(\bfq \cdot \bfn_K, w)_{\partial K} = (\hat \bfq \cdot \hat \bfn_{\hat K},\hat w)_{\partial \hat K}$, 
$( \nabla \cdot \bfq, w)_{K } = (\hat \nabla \cdot \hat \bfq, \hat w)_{\hat K }$, 
and $(\nabla w,\bfq)_K = (\hat \nabla \hat w, \hat \bfq)_{\hat K}$. Since $B$ is a constant 
matrix it follows that  $\bfq \in [\mathbb{P}_{l-1}(K)]^d  \Longleftrightarrow \hat \bfq  \in [\mathbb{P}_{l-1}(\hat K)]^d$,  and thus for $\bfq \in [\mathbb{P}_{l-1}(K)]^d$ we have 
$$(\pi_{\hat K,l-1} \hat{\bfeta}_h,\hat \bfq)_{\hat{K}} 
= 
 (\hat \nabla \hat w, \hat \bfq)_{\hat K} = (\nabla w, \bfq )_{K} 
 =  
 (\pi_{K,l-1} \nabla w,\bfq)_{K} = (\pi_{K,l-1} \bfeta, \bfq)_K.$$
Furthermore, we have the estimates  
\begin{equation}\label{eq:stab-element}
\| \bfq \|^2_K \lesssim |B|^{-1} \| B \|^2 \| \hat \bfq \|^2_{\hat K }, 
\qquad 
 \| \hat \bfq \|^2_{\hat K }  \lesssim  |B|  \| B^{-1} \|^2 \| \bfq \|^2_K.
\end{equation}
Moreover, denoting by $(\Phi|_{\hat F})'$ the derivative of the restriction of $\Phi$ on a face $\hat F$ of the reference element $\hat K$, and by 
$|B|_{\partial K}$ the maximum of $|(\Phi|_{\hat F})'|$
over all the faces $\hat F$, 
it holds that 
\begin{equation}\label{eq:stab-normal-trace}
\| \hat \bfq \cdot \hat \bfn_{\hat K} \|^2_{\partial \hat K} 
\lesssim  |B|_{\partial K} \| \bfq \cdot \bfn_K \|^2_{\partial K}.
\end{equation}
To verify 
(\ref{eq:stab-normal-trace}) we note that 
\begin{equation*}
\sup_{\hat w\in L^2(\hat K)} \frac{(\hat \bfq \cdot \hat \bfn_K,\hat
  w)_{\partial \hat K}}{\| \hat w \|_{\partial \hat K}}
= 
\sup_{w\in L^2(K)} \frac{( \bfq \cdot  \bfn_K,w)_{\partial K}}{\| w \|_{\partial K}} 
\frac{\| w \|_{\partial K}}{\| \hat w \|_{\partial \hat K}}
\lesssim 
\max_{\hat F} |(\Phi|_{\hat F})'| \|  \bfq \cdot  \bfn_K \|_{\partial K}.
\end{equation*}

Recall that we have assumed that the meshes are quasi uniform. In particular, they are shape regular, and therefore the diameter $\rho_K$ of the largest ball in $K$
satisfies $\rho_K \sim h$. The projection of this ball onto the plane containing a face $F$ of $K$ is a $(d-1)$-dimensional ball of the same radius $\rho_K$. But this ball is contained in $F$, and therefore the volume of $F$ is proportional to $h^{d-1}$. This again implies that $|B|_{\partial K} \sim h^{d-1}$. Also,
\begin{equation*}
\| B \| \lesssim h,\qquad \| B^{-1} \|\lesssim h^{-1},\qquad |B| \sim h^{d}.
\end{equation*}
Finally, using (\ref{eq:stab-ref-element}), (\ref{eq:stab-element}), and (\ref{eq:stab-normal-trace}) and the above geometric bounds we obtain
\begin{align*}
\| \bfeta_h \|^2_K &\lesssim  |B|^{-1} \|B \|^2 \|\hat \bfeta_h\|^2_{\hat K}
\\
&\lesssim 
 |B|^{-1} \|B \|^2 ( \|\pi_{\hat K,l-1} \hat \bfeta_h\|^2_{\hat K} + \|\hat \bfeta_h \cdot \hat{\bfn}_K\|^2_{\partial \hat K} ) 
\\
&\lesssim 
 |B|^{-1} \|B \|^2 (
|B| \| B^{-1}\|^2 \|\pi_{K,l-1} \bfeta_h \|^2_{K} 
+ |B|_{\partial K} \|\bfeta_h \cdot {\bfn}_K \|^2_{\partial K} 
)
\\
&\lesssim 
\|\pi_{K,l-1} \bfeta_h \|^2_{K} 
+ h \|\bfeta_h \cdot {\bfn}_K \|^2_{\partial K}.
\end{align*}  
\end{proof}

To measure the effect of the perturbed data we introduce the corrector
function $\delta \bfp \in D_{\delta \psi}^l$,
\begin{equation}\label{eq:lift}
\left<\delta \bfp\cdot \bfn_F , p_h\right>_F = \left\{ \begin{array}{ll}
 \left<\delta \psi,p_h\right>_F& \mbox{ for all } p_h \in
 \mathbb{P}_l(F) \mbox{ for } F \in \mathcal{F}_\Sigma\\
0 & \mbox{ for all } p_h \in
 \mathbb{P}_l(F) \mbox{ for } F \in \mathcal{F} \setminus\mathcal{F}_\Sigma
\end{array} \right.
\end{equation}
and if $k \ge 1$, for any $K \in \mathcal{T}$,
$
(\delta \bfp, \bfq_h)_K = 0, \mbox{ for all } \bfq_h \in [\mathbb{P}_{l-1}(K)]^d.
$
For $\delta \bfp$ we may also show the bound
$
\|\delta \bfp\|_\Omega \lesssim h^{\frac12} \|\delta \psi\|_\Sigma.
$

We will frequently use the following inverse and trace
inequalities, for all $v \in \mathbb{P}_k(K)$,
\begin{equation}\label{eq:inverse}
\|\nabla v\|_K \lesssim h^{-1} \|v\|_K, 
\end{equation}
and for all $v \in H^1(K)$
\begin{equation}\label{eq:trace}
\|v\|_{\partial K} \lesssim h^{-\frac12} \|v\|_K + h^{\frac12}
\|\nabla v\|_K.
\end{equation}
For a proof of \eqref{eq:inverse} we refer to Ciarlet \cite{Ciarlet78}, and
for \eqref{eq:trace} see for instance \cite{MS99}.

\subsection{Deriving finite element methods in an optimization framework}
The method to solve ill-posed problems proposed in
\cite{Bu16} is based on discretization in an optimization framework
where some quantity is minimized under the constraint of the partial
differential equation. The quantity to be minimized is typically
either some least squares fit of data or some weakly consistent
regularization term acting on the discrete space, or both. Introducing
the Lagrange multiplier space $W^m := X_h^m$, this
problem then takes the form of finding the critical point of a
Lagrangian
$\mathcal{L}: V_{g}^k \times D_{\tilde \psi}^l \times W^m \to \mathbb{R}$ defined by
\begin{equation}\label{eq:lagrangian}
\mathcal{L}[v_h,\bfq_h,y_h]:= \frac12 s[(v_h,\bfq_h),(v_h,\bfq_h)] -
\frac12 s^*(y_h,y_h) + b(\bfq_h, v_h,y_h)-(\tilde f,y_h)_\Omega.
\end{equation}
Here $y_h \in W^m$ is the Lagrange multiplier, $s(\cdot,\cdot)$ denotes the primal stabilizer,
$s^*(\cdot,\cdot)$ the dual stabilizer and $b(\cdot,\cdot)$ the
bilinear form defining the partial differential equation, in our case
the conservation law
\[
b(\bfq_h,v_h,y_h):= (\nabla \cdot \bfq_h + \mu v_h, y_h)_\Omega.
\]
As a first step to ensure that the kernel of the
system is trivial we propose the primal stabilizer
\begin{equation}\label{eq:primal_stab}
s[(v,\bfq),(v,\bfq)]:=  \frac12\|A \nabla v -
\bfq\|_\Omega^2 + t(v,v) 
\end{equation}
where $t(\cdot,\cdot)$ is a symmetric positive semi-definite form
related to Tikhonov regularization. However here we will design $t$ so
that it is weakly consistent to the right order.
This should be compared with the jump of the gradient used in
\cite{Bu13}. Observe that in this case the first term of $s$ forces
$\bfp_h$ and $A\nabla u_h$ to be close, connecting the flux variable to the primal variable.
In that way introducing an
effect similar to the penalty on the gradient of \cite{Bu13}. 

Computing the Euler-Lagrange equations of \eqref{eq:lagrangian}
we obtain the following linear system. Find $u_h, \bfp_h, z_h \in  V_{g}^k \times D_{\tilde \psi}^l \times W^m$ such that
\begin{align}\label{eq:EL_1}
s[(u_h,\bfp_h),(v_h,\bfq_h)]+b(\bfq_h, v_h,z_h) & = 0 \\
b(\bfp_h, u_h,w_h) - (\tilde f,w_h)_\Omega  - s^*(z_h,w_h)& = 0\label{eq:EL_2}
\end{align}
for all $v_h, \bfq_h ,w_h \in  V_{0}^k \times D_{0}^l \times W^m$.
The system \eqref{eq:EL_1}-\eqref{eq:EL_2} is of the
same form as that proposed in \cite{Bu14b,Bu16}.
To ensure that the system is well-posed, the spaces $V_{g}^k \times
D_{\tilde \psi}^l \times W^m$ and the stabilizations $t$ and $s^*$
must be carefully balanced.  If we restrict the discussion to $k-1
\leq l \leq k$ and $l \leq m \leq k$, $k \ge 1$, a stable system is
obtained by choosing
\begin{equation}
\begin{cases}
\label{eq:tikhonov}
t(v_h,v_h) := \frac12\mu^2h^2 \|(1 - \pi_W) v_h\|_\Omega^2 + \gamma_T h^{2k} \|\nabla
v_h\|_{\Omega}^2, 
\\ 
s^*(y_h,y_h) :=  \gamma^* \frac12\|(1 - \pi_{X,l-1}) \nabla y_h\|^2_{\Omega},
\end{cases}
\end{equation}
where $\pi_W:L^2(\Omega) \mapsto W^m$ denotes the standard $L^2$-projections on $W^m$. We also define $\pi_{X,-1} \equiv 0$.
Alternatively for any choice of $k,l,m$ one may use the regularizing terms:
\begin{equation}\label{eq:L2LS}
\begin{cases}
t(v_h,v_h) := \frac12\gamma_T h^{2k} \|\nabla v_h\|_{\Omega}^2, 
\\
s^*(y_h,y_h) := \frac12\|y_h\|_\Omega^2.
\end{cases}
\end{equation}

For
simplicity we limit the discussion to $H^1$ conforming approximation
of the primal variable, i.e. $V^k_h \subset H^1(\Omega)$. The extension to non-conforming methods is
immediate, for instance by adding a penalty on the solution jump over
element faces and a penalty on the solution on the boundary $\Sigma$.
The contribution $\|h^{-{\frac12}} \jump{u_h}\|_{\mathcal{F}_i\cup
  \mathcal{F}_\Sigma}^2$ is then added to \eqref{eq:primal_stab}
. Such alternative methods can then be
analysed using the arguments below together with elements of
\cite{Bur17a,Bur17b}. We leave the details to the reader. 

We end this section by detailing some different choices of polynomial
orders for the spaces and associated stabilizers $s,\,s^*$, that result in stable and optimally
convergent methods.
\subsection{Inf-sup stable finite element
  formulation}
If for fixed $k \ge 1$ we take $l = m =k$ then the primal-dual method is
stable with minimal stabilization. It is obvious that for this choice
the first term of $t$ is always zero as well as $s^*$. Considering equation \eqref{eq:EL_2} we see that for every cell
$K \in \mathcal T_h$ we have by taking $w_h = \chi_K$, with $\chi$
denoting the characteristic function, 
\begin{equation}\label{eq:loc_cons}
\int_{\partial K} \bfp_h \cdot \bfn_K ~ \mbox{d}s = \int_K (f - \mu u_h) ~\mbox{d}x
\end{equation}
expressing the cell-wise satisfaction of the conservation law. This
method however has a very large number of degrees of freedom and it is
not obvious how to eliminate the Lagrange multiplier in order to
reduce the size of the system. Moreover the spaces are not matched
with respect to accuracy, optimal estimates are obtained also if
$l=k-1$.
\subsection{Well balanced method with local $H^1$ dual stabilizer}\label{sec:wellbalanced}
For fixed $k$ take $l=k-1$ and $m=k$, then as we shall see below, the
primal and dual spaces are well balanced in the sense that they
produce the same order of approximation error $O(h^k)$ for a
sufficiently smooth solution. Since $V^k_g \subset X_h^m$ the first term
of $t$ in \eqref{eq:tikhonov} is zero. On the other hand with this
choice of spaces the
method is not inf-sup stable for $s^*\equiv 0$. The dual stabilizer in
\eqref{eq:tikhonov} is however completely local to each element. 
In the case $l=0$, the dual
stabilizer \eqref{eq:tikhonov}, becomes
\[
s^*(y_h,y_h) := \frac12\sum_{K \in \mathcal{T}} \|\nabla y_h\|_{K}^2.
\]
Since $s^*$ is zero for constant functions the relation
\eqref{eq:loc_cons} still holds. Thanks to the local character, all the
degrees of freedom of the Lagrange multiplier, except the cell-wise
average value, can be eliminated from the system using static
condensation. 

For the choice $l=k-1$ and $m=k-1$, we take $t$ defined by
\eqref{eq:tikhonov} and $s^*\equiv 0$, which results
in an inf-sup stable well balanced method. If $\mu=0$, the first term
of $t$ can be omitted, i.e. $t(v_h,v_h):=\frac12\gamma_T h^{2k} \|\nabla v_h\|_{\Omega}^2$.  This method can
easily be analysed using the approach below and has similar
convergence order as the previous well-balanced method.
\subsection{mixed $L^2$-least squares finite element formulation}
The choice of spaces and stabilizers proposed above lead to methods
that have optimal convergence properties up to physical stability and that 
satisfy the conservation law exactly on each cell. These properties
however come at a price: the number of degrees of freedom is
large. Indeed compared to the method introduced in \cite{Bu13}, using
piecewise affine conforming approximation for both the primal and dual
variable the number of degrees of freedom increases at least by a
factor of three if this formulation is used. This large increase can be
reduced to a factor of two by using the dual stabilizer \eqref{eq:L2LS}
as we shall see below, but the price is that local
conservation only holds weakly.

If we define $s^*(z_h,w_h):= (z_h,w_h)_\Omega$ we immediately get from \eqref{eq:EL_2} that $z_h = \nabla
\cdot \bfp_h + \mu u_h - f_h$ where $f_h$ is the $L^2$-projection of $f$
onto $W^m$. Reinjecting this expression for $z_h$ into \eqref{eq:EL_1}
and defining $s$ by \eqref{eq:primal_stab}, with 
$t$ as in
\eqref{eq:L2LS},
we obtain the equation: find $(u_h,\bfp_h) \in V_g^k \times D_{\tilde\psi}^l$ such that 
\begin{equation}\label{eq:L2_LS}
s[(u_h,\bfp_h),(v_h,\bfq_h)]+(\mu u_h+\nabla \cdot \bfp_h,\mu v_h+\nabla
\cdot \bfq_h)_\Omega  = (\tilde f, \nabla
\cdot \bfq_h+ \mu v_h)_\Omega
,
\end{equation}
for all $(v_h,\bfq_h) \in V_0^k \times D_0^l$.
This method, which coincides with the one proposed in \cite{DHH13} up to
Tikhonov regularization, can be derived directly from the minimization of
the following functional 
$\mathcal{J}_h:V_{g}^k \times D_{\tilde \psi}^l \to \mathbb{R}$,
\begin{equation*}
\mathcal{J}_h(u_h,\bfp_h) := s[(u_h,\bfp_h),(u_h,\bfp_h)] +
\int_\Omega
(\nabla \cdot \bfp_h + \mu u_h- \tilde f)^2 ~\mbox{d}x.
\end{equation*}
There is one Tikhonov regularization term
added in $s$, where the parameter $\gamma_T$ is independent of the mesh
size. Our discrete method can now be written: find $(u_h, \bfp_h) \in
V_{g}^k \times D_{\tilde \psi}^l$ such that
\begin{equation}\label{L2discmin}
(u_h,\bfp_h) = \argmin_{ 
V_{g}^k \times D_{\tilde \psi}^l} \, \mathcal{J}_h(u_h,\bfp_h).
\end{equation}
We conclude that the solution of \eqref{eq:EL_1}-\eqref{eq:EL_2}
coincides with the minimizer of \eqref{L2discmin} for the dual
stabilizer of the left relation of \eqref{eq:L2LS}. Compared
to the method proposed in \cite{DHH13} the regularization has been
reduced to only one term. Indeed this term is all that we need to
prove optimal error estimates and its only role is to ensure a
uniform apriori estimate on the discrete solution in $H^1$. 
As we shall see below, an 
iteration based on the method \eqref{eq:L2_LS} can be used to solve one of the previous, larger systems, thus recovering the local conservation and optimal error estimates.
\subsection{Stability and continuity of the forms}
For the analysis we introduce norms on
$V_{\Sigma} \times H_{div}(\Omega)$,
\begin{align*}
\tnorm{(v,\bfq)}_{-\zeta} &:= \left(
s[(v,\bfq),(v,\bfq)] + \|h^{\zeta}(\nabla \cdot
\bfq+ \mu v)\|_\Omega^2 \right)^{\frac12},
\\
\tnorm{(v,\bfq)}_{\sharp} &:=\tnorm{(v,\bfq)}_{-\zeta}+\mu\|v\|_\Omega+\|h^{\frac12} \bfq\|_{\mathcal{F}}+\|\bfq\|_{\Omega},
\end{align*}
where, depending on the choice of the spaces and stabilizers, either $\zeta = 0$ or $\zeta = 1$.
Using the approximation properties \eqref{eq:approx_Lagrange},
\eqref{eq:approx_RT} and the trace inequality \eqref{eq:trace} it is
straightforward to prove the following approximation result for the
triple norms, we omit the details,
\begin{equation}\label{eq:tnorm_approx}
\tnorm{(v - i_h
  v,\bfq - \bfR_h \bfq_h)}_{\sharp} \lesssim h^k \|u\|_{H^{k+1}(\Omega)}
+ h^{l+1} (|\bfq|_{H^{l+1}(\Omega)}+|\nabla \cdot \bfq|_{H^{l+1-\zeta}(\Omega)}).
\end{equation}
The system \eqref{eq:EL_1}-\eqref{eq:EL_2} can be written on the compact form:
 find $(u_h,\bfp_h,z_h) \in V_{g}^k \times
D_{\tilde \psi}^l \times W^m$ such that
\begin{equation}\label{eq:compact}
\mathcal{A}[(u_h,\bfp_h,z_h),(v_h,\bfq_h,w_h)] = l(w_h),\quad  \forall
(v_h,\bfq_h,y_h) \in V_{0}^k \times D_{0}^l \times W^m,
\end{equation}
where
$$
\mathcal{A}[(u_h,\bfp_h,z_h),(v_h,\bfq_h,y_h)]:= b(\bfq_h, v_h,z_h)+
b(\bfp_h, u_h,y_h)
-s^*(z_h,y_h)+s[(u_h,\bfp_h),(v_h,\bfq_h)],
$$
with $s$ and $s^*$ given by \eqref{eq:primal_stab} and
\eqref{eq:tikhonov}, and the right hand side given by
\[
 l(w_h) := (\tilde f, w_h)_\Omega.
\]
We will also
use the following compact notation for the reduced method \eqref{eq:L2_LS}: find $(u_h,\bfp_h) \in V_g^k \times
D_{\tilde \psi}^l$ such that
\begin{equation}\label{eq:compact_red}
{\mathcal{A}}_R[(u_h,\bfp_h,z_h),(v_h,\bfq_h,x_h)] = l_R(\bfq_h),\quad  \forall (v_h,\bfq_h) \in V_{0}^k \times D_{0}^l,
\end{equation}
where, 
\begin{equation}\label{eq:def_AR}
{\mathcal{A}}_R[(u_h,\bfp_h),(v_h,\bfq_h)]:=
(\nabla \cdot \bfp_h+ \mu u_h, \nabla\cdot \bfq_h+ \mu v_h )_\Omega
+ s[(u_h,\bfp_h),(v_h,\bfq_h)],
\end{equation}
with $s$ defined by \eqref{eq:primal_stab} and \eqref{eq:L2LS},
and
\[
 l_R(\bfq_h,v_h) := (\tilde f, \nabla \cdot \bfq_h+ \mu v_h)_\Omega.
\]
Observe that for the exact solution $(u,\bfp)$ there
holds
\begin{equation}\label{eq:consist_full}
\mathcal{A}[(u,\bfp,0),(v_h,\bfq_h,w_h)] = l(w_h) - (\delta f,
w_h)_\Omega + t(u,v_h)
\end{equation}
and, similarly for the reduced method,
\begin{equation*}
\mathcal{A}_R[(u,\bfp),(v_h,\bfq_h)] = l_R(\bfq_h,v_h) - (\delta f, \nabla \cdot
\bfq_h+ \mu v_h)_\Omega + t(u,v_h).
\end{equation*}
 We will
now prove a stability result that is the cornerstone of both the
methods. The method \eqref{eq:compact} requires
an inf-sup argument and the symmetric method \eqref{eq:compact_red} is coercive.
\begin{prop}\label{infsup}
For the formulation \eqref{eq:compact} with $k-1\leq l \leq m$ and $l
\leq m \leq k$ and, when $l<m$,  $\gamma^*>0$ small enough,  there
exists $\alpha>0$ such that for all $v_h,\bfq_h,x_h \in V_{0}^k \times D^l_{0} \times W^m$ there
exists $w_h,\bfy_h,r_h \in V_{0}^k \times D^l_{0} \times
W^m$ such that
\begin{equation}\label{eq:is1}
\alpha(\tnorm{(v_h,\bfq_h)}_{-1}^2 + \|x_h\|_{1,h}^2) \leq \mathcal{A}[(v_h,\bfq_h,x_h),(w_h,\bfy_h,r_h)]
\end{equation}
and
\begin{equation}\label{eq:is2}
\tnorm{(w_h,\bfy_h)}_{-1} + \|r_h\|_{1,h} \lesssim \tnorm{(v_h,\bfq_h)}_{-1} + \|x_h\|_{1,h}.
\end{equation}
For the reduced method \eqref{eq:compact_red} the following coercivity holds. For all $(v,\bfq) \in H^1(\Omega) \times
H_{div}(\Omega)$
\begin{equation}\label{eq:coerce}
\tnorm{(v,\bfq)}_0^2 = \mathcal{A}_R[(v,\bfq),(v,\bfq)].
\end{equation}
\end{prop}
\begin{proof}
The relation \eqref{eq:coerce} is immediate by the definition of
$\mathcal{A}_R$ \eqref{eq:def_AR}. We now consider the first claim.
Let $\xi_h :=h^2 (\nabla \cdot \bfq_h + \mu \pi_W v_h) \in W^m$. Then
\begin{align}\nonumber
b(\bfq_h, v_h, \xi_h) &= (\nabla \cdot \bfq_h
+ \mu  v_h, h^2(\nabla \cdot \bfq_h + \mu \pi_W  v_h))_\Omega 
\\ \nonumber
&\ge \frac12 \|h (\nabla \cdot \bfq_h + \mu  v_h)\|_\Omega^2 -
\frac12 \mu^2 h^2 \|(1 - \pi_W)v_h\|^2,
\end{align}
and using Cauchy-Schwarz inequality, the stability of the
$L^2$-projection and the inverse inequality \eqref{eq:inverse},
\[
s^*(x_h,x_h-\xi_h) \ge \frac12 s^*(x_h,x_h) - \frac12 \gamma^* C_i \|h (\nabla \cdot \bfq_h + \mu  v_h)\|_\Omega^2.
\]
It follows from the above bounds that assuming $\gamma^*< (2
C_i)^{-1}$ there holds
\begin{multline}\label{eq:est_primvar}
\frac14 \|h (\nabla \cdot \bfq_h + \mu  v_h)\|_\Omega^2 + \|A \nabla v_h - \bfq_h\|^2_\Omega
+ \frac12 t(v_h,v_h)+\frac12 s^*(x_h,x_h) \\\leq \mathcal{A}[(v_h,\bfq_h,x_h),(v_h,\bfq_h,-x_h+\xi_h)].
\end{multline}
We recall that when $l<m$, the dual stabilizer $s^*$ is defined by the second equation of
\eqref{eq:tikhonov}  and when $l=m$, $s^*\equiv 0$.

To prove stability of the multiplier, that is, to obtain the term $ \|x_h\|_{1,h} $ on the left
hand side of \eqref{eq:is1}, we consider the test function
$\bfeta_h = \bfeta_h(x_h)$ as defined in \eqref{eq:reconstruct1}-\eqref{eq:reconstruct2}.
It then follows by the definition of $\mathcal{A}$ that
\[
 \mathcal{A}[(v_h,\bfq_h,x_h),(0,\bfeta_h,0)] = (A \nabla v_h - \bfq_h, - \bfeta_h)_\Omega + (\nabla \cdot \bfeta_h, x_h )_\Omega.
\]
Using elementwise integration by parts in the second term on the right
hand side yields, 
\[
(\nabla \cdot \bfeta_h, x_h )_\Omega = \sum_{K\in \mathcal{T}} \left[ \left<\bfeta_h \cdot \bfn_K,
x_h\right>_{\partial K}-(\bfeta_h, \nabla
x_h)_K \right].
\]
For the second term of the right hand side we obtain using \eqref{eq:reconstruct2},
\[
-(\bfeta_h, \nabla
x_h)_K \ge  \|\pi_{X,l-1}\nabla x_h\|^2_K - \frac{\gamma^*}{4}
\|(1 - \pi_{X,l-1})\nabla x_h\|_K^2 -\frac{1}{ \gamma^*}
\|\bfeta_h\|_K^2.
\]
Observe that by combining the contributions from the two
neighbouring elements sharing a face $F$ we have
\[
\sum_{K\in \mathcal{T}} \left<\bfeta_h \cdot \bfn_K,
x_h\right>_{\partial K} = \sum_{F\in \mathcal{F}\setminus \mathcal{F}_\Sigma}\|h^{-\frac12}
\piF \jump{x_h}\|^2_{F},
\]
where we used (\ref{eq:reconstruct1}) and the fact that $\bfeta_h \cdot \bfn_K= 0$ on $\Sigma$. Consequently,
\begin{align*}
 \mathcal{A}[(v_h,\bfq_h,x_h),(0,\epsilon \bfeta_h,0)] 
 &\ge (A \nabla v_h - \bfq_h, -
 \epsilon \bfeta_h)_\Omega +  \epsilon \|\pi_{X,l-1}\nabla x_h\|^2_\Omega  
 \\
&
+ \epsilon \sum_{F\in \mathcal{F}\setminus \mathcal{F}_\Sigma}\|h^{-\frac12}
\piF \jump{x_h}\|^2_{F} 
- \frac{\gamma^*}{4}
\|(1 - \pi_{X,l-1})\nabla x_h\|_\Omega^2 -\frac{\epsilon^2}{\gamma^*}
\|\bfeta_h\|_\Omega^2.
\end{align*}
We obtain
\begin{align*}
&\mathcal{A}[(v_h,\bfq_h,x_h),(0,\epsilon \bfeta_h,0)] 
\\&\quad\ge - \frac{1}{4} \|A \nabla v_h - \bfq_h\|^2_\Omega-
  \epsilon^2\left(1 + \frac{1}{\gamma^*} \right) \|\bfeta_h\|^2_\Omega 
  +  \epsilon \|\pi_{X,l-1}\nabla x_h\|^2_\Omega   
 + \epsilon \sum_{F\in \mathcal{F}\setminus \mathcal{F}_\Sigma}\|h^{-\frac12}
\piF \jump{x_h}\|^2_{F}  
\\ 
&\qquad - \frac{\gamma^*}{4}
\|(1 - \pi_{X,l-1})\nabla x_h\|_\Omega^2 
\\
&\quad\ge  
   - \frac{1}{4}  \|A \nabla v_h - \bfq_h\|^2_\Omega  - \frac{\gamma^*}{4}
\|(1 - \pi_{X,l-1})\nabla x_h\|_\Omega^2\\&\qquad
 +  \epsilon (1 - \epsilon (1+{\gamma^*}^{-1}) C_{ds}^2 )  \left(\|\pi_{X,l-1}\nabla x_h\|^2_\Omega +
\sum_{F\in \mathcal{F}\setminus \mathcal{F}_\Sigma}\|h^{-\frac12}
\piF \jump{x_h}\|^2_{F} \right).
\end{align*}
Here we used the stability \eqref{eq:etastab} of $\bfeta_h$.
Choosing $\epsilon = C_{ds}^{-2} 2^{-1} \gamma^* (1+ \gamma^*)^{-1} $ we see that
\begin{align}\label{eq:est_dualvar}
&\mathcal{A}[(v_h,\bfq_h,x_h),(0,\epsilon \bfeta_h,0)] 
\\\notag&\quad\ge -\frac14 \|A \nabla v_h
- \bfq_h\|^2_\Omega
 +  \frac{\epsilon}{2}  \left(\|\pi_{X,l-1}\nabla x_h\|^2_\Omega +
\sum_{F\in \mathcal{F}\setminus \mathcal{F}_\Sigma}\|h^{-\frac12}
\piF \jump{x_h}\|^2_{F} \right)
\\ \nonumber 
&\qquad
-  \frac{\gamma^*}{4}
\|(1 - \pi_{X,l-1})\nabla x_h\|_\Omega^2.
\end{align}
By combining the bounds \eqref{eq:est_primvar} and
\eqref{eq:est_dualvar}, and using that
\[
\|x_h\|_{1,h} = \|\pi_{X,l-1}\nabla x_h\|^2_\Omega +\|(1 -
\pi_{X,l-1})\nabla x_h\|_\Omega^2+ 
\sum_{F\in \mathcal{F}\setminus \mathcal{F}_\Sigma}\|h^{-\frac12}
\piF \jump{x_h}\|^2_{F},
\] 
we obtain
\[
\frac14 \tnorm{(v_h,\bfq_h)}^2 +  \frac14 \min(\epsilon, \gamma^*) 
\|x_h\|_{1,h}^2
\leq 
\mathcal{A}[(v_h,\bfq_h,x_h),(v_h,\bfq_h+ \epsilon \bfeta_h,-x_h+ \xi_h)],
\]
which proves \eqref{eq:is1}, with $\alpha = 1/2\min(1, \epsilon,\gamma^*)$ and with the test partners $w_h = v_h,\, \bfy_h = \bfq_h+
\epsilon \bfeta_h$ and $r_h = -x_h+ \xi_h$.

For the second
inequality \eqref{eq:is2}, we observe that by the triangle inequality there holds
\[
\tnorm{(w_h,\bfy_h)}_{-1} + \|r_h\|_{1,h} \leq
\tnorm{(v_h,\bfq_h)}_{-1}+\|x_h\|_{1,h}+\tnorm{(0,\epsilon \bfeta_h)}_{-1}+\|\xi_h\|_{1,h}.
\]
To bound the second to last term on the right hand side,
we use the 
inverse inequality 
\[
\tnorm{(0,\epsilon \bfeta_h)}_{-1} = \epsilon ( \|\bfeta_h\|_\Omega
+\|h \nabla \cdot \bfeta_h\|_\Omega) \lesssim
\|\bfeta_h\|_\Omega  \lesssim \|x_h\|_{1,h}.
\]
Using an inverse inequality \eqref{eq:inverse} and a trace inequality \eqref{eq:trace} in the last term of the right hand side, we obtain
\[
\|\xi_h\|_{1,h}  \lesssim \|h (\nabla \cdot
\bfq_h + \mu v_h)\|_{\Omega}+h \mu \|(1-\pi_W) v_h\|_{\Omega}.
\]
Since it follows that $\tnorm{(w_h,\bfy_h)}_{-1} + \|r_h\|_{1,h}
\lesssim \tnorm{(v_h,\bfq_h)}_{-1} + \|x_h\|_{1,h}$ the proof is complete.
\end{proof}

Using the previous result, we now show that the discrete solution
will exist, regardless of the choice of the parameter
$\gamma_T \ge 0$. 
\begin{prop}(Invertibility of system matrix.)
The linear system defined by \eqref{eq:compact}, with spaces and dual
stabilizations as in Proposition \ref{infsup},
admits a unique solution $(u_h ,\bfp_h,z_h)$ in $V_{g}^k \times
D_{\tilde \psi}^l \times W^m$. The linear system defined by \eqref{eq:compact_red} 
admits a unique solution $(u_h ,\bfp_h)$ in $V_{g}^k \times
D_{\tilde \psi}^k$.
\end{prop}
\begin{proof}
%
Since existence and uniqueness are equivalent for square, finite dimensional
linear systems we only need to show uniqueness. 
We consider a difference $(u_h,\bfp,z_h) \in V_{0}^k \times
D_{0}^l \times W^m$ of two solutions, and show that it is 
zero if
\begin{equation*}
\mathcal{A}[(u_{h},\bfp_{h},z_h),(v_h,\bfq_h,x_h)] = 0 \quad \forall
(v_h,\bfq_h,x_h) \in V_{0}^k \times D_{0}^l \times W^m.
\end{equation*}

By Proposition \ref{infsup} there then holds, 
\[
\alpha (\tnorm{(u_{h},\bfp_{h})}^2_{-1} + \|z_h\|^2_{1,h} ) \leq
\mathcal{A}[(u_{h},\bfp_{h},z_h),(w_h,\bfy_h,r_h)] = 0
\]
and we immediately see that $z_h = 0$.
In the case $\gamma_T>0$
the equation $\tnorm{(u_{h},\bfp_{h})}^2_{-1} = 0$ implies the claim since we obtain $\|\nabla u_{h}\|_\Omega =
\|\bfp_{h}\|_\Omega = 0$ , and the conclusion follows after noting that
the $H^1$-semi norm on $V_{0}^k$ is a norm by the
Poincar\'e inequality.

Assume now that $\gamma_T = 0$. In this case the stability implies 
\[
\|A \nabla u_{h} - \bfp_{h}\|_\Omega^2 + \|\nabla \cdot \bfp_{h}+\mu u_{h}\|_\Omega^2 = 0.
\]
This means that $A \nabla u_{h} = \bfp_{h}$ and $\nabla \cdot \bfp_{h}+\mu u_{h} =
0$. As a consequence $\nabla \cdot (A \nabla u_h) \in L^2(\Omega)$, $u_{h}\vert_\Sigma = A \nabla u_{h} \cdot
\bfnu  \vert_\Sigma = 0$ and
\[
\nabla \cdot (A \nabla u_{h}) + \mu u_{h} = 0 \mbox{ in } \Omega.
\]
It follows that $u_{h}$ is a solution to the problem
\eqref{weak_Cauchy} with zero data. 
The stability estimate \eqref{eq:globbound} implies
that the trivial solution $u_h = 0$ is the unique solution of this 
problem. It follows that $u_{h}=0$ and $\bfp_{h} =
0$. This proves the claim. 
The claimed uniqueness for 
\eqref{eq:compact_red} is immediate due to the coercivity.
\end{proof}

We end this section by proving the continuity of the forms
$\mathcal{A}[\cdot,\cdot]$. 
\begin{proposition}
For all $(v,\bfq) \in H^1(\Omega) \times
H_{div}(\Omega)$ and for all $(w_h,\bfy_h,w_h)$  there holds,
\begin{equation}\label{eq:cont1}
\mathcal{A}[(v,\bfq,0), (w_h,\bfy,w_h)] \leq \tnorm{(v,\bfq)}_{\sharp} \, (\tnorm{(w_h,\bfy_h)}_{-1}+\|w_h\|_{1,h}).
\end{equation}
For all $(v,\bfq), (w,\bfy) \in H^1(\Omega) \times
H_{div}(\Omega)$ there holds
\begin{equation}\label{eq:cont2}
\mathcal{A}_R[(v,\bfq), (w,\bfy)] \leq \tnorm{(v,\bfq)}_0 \, \tnorm{(w,\bfy)}_0.
\end{equation}
\end{proposition}
\begin{proof}
The inequality \eqref{eq:cont1} follows by first using the Cauchy-Schwarz
inequality in the symmetric part of the formulation,
$
s[(v,\bfq),(w_h,\bfy)]
 \leq s[(v,\bfq),(v,\bfq)]^{\frac12} s[(w_h,\bfy),(w_h,\bfy)]^{\frac12}.
$
In the remaining term we use the divergence formula elementwise to obtain
\[
(\nabla \cdot \bfq + \mu v, w_h)_\Omega = \sum_{K \in \mathcal{T}}
(\bfq,\nabla w_h)_K + \sum_{F \in \mathcal{F}} \left<\bfq \cdot \bfn_F,\jump{w_h}
\right>_F+(\mu v,w_h)_\Omega.
\]
The inequality now follows by applying the Cauchy-Schwarz inequality termwise
with suitable scaling in $h$.
The inequality \eqref{eq:cont2} on the other hand is immediate by applying the
Cauchy-Schwarz inequality to the form $\mathcal{A}_R$ that is
completely symmetric in this case.
\end{proof}
\section{Error estimates using conditional stability} 
In this section we will prove error estimates that give, for
unperturbed data, an optimal convergence order with respect to
the approximation and stability properties of the problem. We also
quantify the effect of perturbations in data and the resulting possible growth
of error under refinement. Throughout this section we assume that
spaces and parameters in the methods are chosen so that Proposition
\ref{infsup} holds.
\begin{prop}(Estimate of residuals.)\label{prop:residualest}
Assume that $(u,\bfp)$ is the solution to \eqref{weak_Cauchy}, where
$\bfp = A\nabla u$ and consider either $(u_h,\bfp_h,z_h)$ the solution of
\eqref{eq:compact} or
$(u_h,\bfp_h)$ the solution of \eqref{eq:compact_red}. Then there holds
\[
\tnorm{(u-u_h,\bfp - \bfp_h)}_{-\zeta}+ \zeta \|z_h\|_{1,h} \lesssim  C_u
h^{k}+C_{\bfp} h^{l+1} + \|\delta f\|_\Omega + h^{-\frac12+\zeta} \|\delta \psi\|_\Sigma,
\]
 with $\zeta=1$ for the method \eqref{eq:compact} and $\zeta=0$ for the method
 \eqref{eq:compact_red}. Here
    \begin{equation*}
C_u:=|u|_{H^{k+1}(\Omega)}+\gamma_T^{\frac12} \|u\|_{H^{1}(\Omega)},
\quad
C_{\bfp} := \|\nabla \cdot \bfp\|_{H^{l+1-\zeta}(\Omega)}+ \|\bfp\|_{H^{l+1}(\Omega)}.
    \end{equation*}
\end{prop}
\begin{proof}
We write the errors in the primal and flux variable 
\[
e = u - u_h \mbox{ and } \bfxi = \bfp - \bfp_h.
\]
Using the nodal interpolant $i_h u$ and the Raviart-Thomas
interpolant $\bfR_h \bfp$ we decompose the error in the interpolation
error $e_\pi:= u-i_h u$, $\bfxi_\pi = \bfp - \bfR_h \bfp$ and the
discrete error,
$e_h = i_h u - u_h \in V_0^k$, $\bfxi_h = \bfR_h \bfp + \delta \bfp -
\bfp_h \in D_0^l$,
where $\delta \bfp$ is defined by equation \eqref{eq:lift}. Observe
that since $e = e_\pi + e_h$ and $\bfxi
= \bfxi_\pi + \bfxi_h - \delta \bfp$, by the triangle inequality there holds
\begin{equation}\label{eq:tnorm_triang}
\tnorm{(e,\bfxi)}_{-\zeta}\leq
\tnorm{(e_\pi,\bfxi_\pi)}_{-\zeta}+\tnorm{(e_h,\bfxi_h)}_{-\zeta}+\tnorm{(0,\delta
  \bfp)}_{-\zeta}.
\end{equation}

We begin with method \eqref{eq:compact}, $\zeta =1$.
Since $e_h \in V_{0}^k$ and $\bfxi_h \in D_{0}^l$
we may apply the stability
result of Proposition \ref{infsup}. Therefore there exists
$(w_h,\bfy_h,r_h) \in V_{0}^k\times D_{0}^l \times W^m$ such that
\[
\alpha (\tnorm{(e_h,\bfxi_h)}_{-1}^2+\|z_h\|_{1,h}^2) \leq \mathcal{A}[(e_h,\bfxi_h,z_h),(w_h,\bfy_h,r_h)],
\]
and also (\ref{eq:is2}) holds.
Now by \eqref{eq:consist_full},
\begin{align*}
& \mathcal{A}[(e_h,\bfxi_h,z_h),(w_h,\bfy_h,r_h)] 
 \\[3mm]
&=  \mathcal{A}[(i_h
 u,\bfR_h \bfp,0),(w_h,\bfy_h,r_h)]+\mathcal{A}[(0,\delta
 \bfp,0),(w_h,\bfy_h,r_h)]-\mathcal{A}[(u_h,\bfp_h,z_h),(w_h,\bfy_h,r_h)]
 \\[3mm]
&=  \mathcal{A}[(i_h
 u,\bfR_h \bfp,0),(w_h,\bfy_h,r_h)] - (\delta \bfp, A \nabla
 w_h-\bfy_h)_\Omega + (\nabla \cdot \delta  \bfp, r_h)-(f,r_h)_\Omega -
 (\delta f, r_h)_\Omega 
 \\[3mm]
& = \underbrace{\mathcal{A}[(i_h u-u,\bfR_h \bfp-\bfp,0),(w_h,\bfy_h,r_h)]}_{I} - \underbrace{(\delta \bfp, A \nabla
 w_h-\bfy_h)_\Omega}_{II} 
 \\
&\qquad + \underbrace{(\nabla \cdot \delta  \bfp,
 r_h)_\Omega}_{III}
-\underbrace{ (\delta f, r_h)_\Omega}_{IV}+\underbrace{ t(u,w_h)}_{V}
\\[3mm]
&=I + II + III+IV+V.
\end{align*}
We now bound the five terms. By the continuity \eqref{eq:cont1} there
holds
\[
I \leq \tnorm{(i_h u - u, \bfR_h \bfp - \bfp)}_{\sharp}(\tnorm{(w_h,\bfy_h)}_{-1}+ \|r_h\|_{1,h}).
\]
An application of the Cauchy-Schwarz inequality leads to 
\[
II \leq \|\delta \bfp\|_{\Omega} \tnorm{(w_h,\bfy_h)}_{-1}.
\]
Finally an elementwise application of the divergence theorem followed
by the Cauchy-Schwarz inequality leads to 
\[
III \leq \tnorm{(0, \delta \bfp) }_{\sharp} \|r_h\|_{1,h}.
\]
By applying the Poincar\'e inequality for broken $H^1$-spaces
\cite{Brenner03} we have for term $IV$,
\[IV \lesssim \|\delta f\|_{\Omega}  \|r_h\|_{1,h}.
\]
Finally, by the Cauchy-Schwarz inequality, using that $m \ge k-1$  and
the standard approximation
estimates for the $L^2$-projection we have
\[
V \leq (\mu h^{k+1}|u|_{H^{k+1}(\Omega)} + \gamma_T^{\frac12} h^k |u|_{H^1(\Omega)}) \tnorm{(w_h,0)}_{-1}.
\]
By (\ref{eq:is2}) we obtain
\[
\alpha (\tnorm{(e_h,\bfxi_h)}_{-1}+\|z_h\|_{1,h}) \lesssim \tnorm{(i_h u - u, \bfR_h \bfp - \bfp)}_{\sharp}+\gamma_T^{\frac12} h^k |u|_{H^1(\Omega)} +\tnorm{(0, \delta \bfp) }_{\sharp}+\|\delta f\|_{\Omega}.
\]
Since the first term on the right hand side is bounded by
\eqref{eq:tnorm_approx} it only remains to show that
\[
\tnorm{(0, \delta \bfp) }_{\sharp} \lesssim h^{\frac12} \|\delta \psi\|_{\Sigma}.
\]
This relation can be proven by the trace inequality and the inverse inequality 
followed by the properties of the Raviart-Thomas element,
\[
\tnorm{(0, \delta \bfp) }_{\sharp} = \|h^{\frac12} \delta \bfp\|_{\mathcal{F}}+2\|\delta \bfp\|_{\Omega} + \|h \nabla\cdot
\delta \bfp\|_\Omega
 \lesssim  \|h^{\frac12} \delta \psi\|_{\Sigma}+\|\delta \bfp\|_{\Omega} \lesssim
\|h^{\frac12} \delta \psi\|_{\Sigma}.
\]
Applying \eqref{eq:tnorm_approx} and $\tnorm{(0,\delta \bfp)}_{-1} \leq
\tnorm{(0,\delta \bfp)}_{\sharp}$ the claim follows from
\eqref{eq:tnorm_triang}.

Let us now turn to method \eqref{eq:compact_red}, $\zeta=0$. This case is similar
  to the previous one, but simpler since it relies on the coercivity
  \eqref{eq:coerce}. Starting from
  \eqref{eq:tnorm_triang} we see that using previous results, there
  only remains to treat the discrete error term. By \eqref{eq:coerce}
  there holds
\[
\tnorm{(e_h,\bfxi_h)}_0^2 \leq \mathcal{A}_R[(e_h,\bfxi_h),(e_h,\bfxi_h)].
\]
Repeating the previous consistency argument, but this time with
$\mathcal{A}_R$ we obtain
\begin{align*}
&\mathcal{A}_R[(e_h,\bfxi_h),(e_h,\bfxi_h)] 
\\[3mm]
&\qquad =  \mathcal{A}_R[(i_h
 u,\bfR_h \bfp),(e_h,\bfxi_h)]+\mathcal{A}_R[(0,\delta
 \bfp),(e_h,\bfxi_h)]-\mathcal{A}_R[(u_h,\bfp_h),(e_h,\bfxi_h)] 
 \\[3mm]
&\qquad =  \mathcal{A}_R[(i_h
 u,\bfR_h \bfp),(e_h,\bfxi_h)] 
 - (\delta \bfp, A \nabla e_h-\bfxi_h)_\Omega 
 + (\nabla \cdot \delta  \bfp, \nabla \cdot \bfxi_h + \mu e_h)_\Omega
  \\[3mm]
 &\qquad \qquad 
 - (f, \nabla \cdot \bfxi_h + \mu e_h)_\Omega 
  -  (\delta f, \nabla \cdot
 \bfxi_h + \mu e_h)_\Omega
 \\[3mm]
&\qquad = \mathcal{A}_R[(i_h
 u - u,\bfR_h \bfp-\bfp),(e_h,\bfxi_h)] - (\delta \bfp, A \nabla
 e_h-\bfxi_h)_\Omega 
  \\[3mm]
 &\qquad \qquad 
 + \underbrace{(\nabla \cdot \delta  \bfp, \nabla \cdot
 \bfxi_h + \mu e_h)_\Omega}_{*}
- (\delta f, \nabla \cdot
 \bfxi_h + \mu e_h)_\Omega+\gamma_T
(h^{2k} \nabla u,\nabla e_h)_\Omega
\end{align*}
The only term we need to consider this time is the one marked $*$. All
the other terms are handled similarly as before, but this time using
\eqref{eq:cont2} and recalling that $\zeta=0$ in
\eqref{eq:tnorm_approx}.
For the term marked $*$ we can not use the divergence formula since
the multiplier has been eliminated. Instead we proceed with the
Cauchy-Schwarz inequality followed by the inverse inequality
\eqref{eq:inverse} and the properties of $\delta \bfp$,
\[
(\nabla \cdot \delta  \bfp, \nabla \cdot
 \bfxi_h + \mu e_h)_\Omega \leq \|\nabla \cdot \delta \bfp\|_\Omega
 \tnorm{(e_h,\bfxi_h)}_0 \lesssim \|h^{-\frac12} \psi\|_{\Sigma}\tnorm{(e_h,\bfxi_h)}_0.
\]
The claim then follows in the same way as before.
\end{proof}
\begin{rem}
To balance the estimate of Proposition \ref{prop:residualest} we
want to balance the orders $O(h^k)$ and $O(h^{l+1})$ to obtain an economical scheme, implying that
$l=k-1$. But we should also balance the regularity
requirements, recalling that $\bfp = A \nabla u$, leading to $k+1 =
l+3-\zeta$.
We see that this can only be balanced for $\zeta=1$. We conclude that
the only method that balances both the convergence orders and the regularities of the
different terms is the one discussed in Section
\ref{sec:wellbalanced}, i.e. the one given by \eqref{eq:compact}.
\end{rem}
\begin{cor}(A priori estimate for the $H^1$-error)\label{cor:H1apriori}
Suppose that $\gamma_T > 0$. Under the same assumptions as for Proposition \ref{prop:residualest}
there holds
\[
\|u-u_h\|_{H^1(\Omega)} \lesssim  \gamma_T^{-\frac12} (C_u + C_{\bfp}
h^{l+1-k} + h^{-k} \|\delta
f\|_{\Omega} + h^{-\frac12-k+\zeta} \|\delta \psi\|_\Sigma),
\]
where $C_u$ and $C_{\bfp}$ are defined in Proposition \ref{prop:residualest}.
\end{cor}
\begin{proof}
By the definition of the triple norm there holds
\[
\gamma^{\frac12}_T h^k \|\nabla (u - u_h)\|_\Omega \lesssim C_u
h^{k}+C_{\bfp} h^{l+1}+\|\delta f\|_\Omega+ h^{-\frac12+\zeta} \|\delta \psi\|_\Sigma.
\]
Then divide through by $\gamma^{\frac12}_T h^k$ and apply 
Poincar\'e's inequality.
\end{proof}

\begin{rem}\label{rem:low_order}
In the lowest order case, $k=1$, if $A$ is the identity matrix, the a priori bound can be achieved
also for $\gamma_T=0$. 
\end{rem}

Indeed observe that, with $e_h=i_h u - u_h$ and
$\bfxi_h = \bfR_h \bfp + \delta \bfp -
\bfp_h$ we have, using the Poincar\'e inequality
on discrete spaces:
\begin{align*}
\|h \nabla e_h\|_\Omega 
&\lesssim 
\|h^{\frac12} \jump{\nabla e_h \cdot \bfn}\|_{\mathcal{F} \setminus \mathcal{F}_\Sigma} 
+ \|h^{\frac12}\nabla e_h \cdot \bfn\|_{\Sigma} 
+ \underbrace{\|h^{-\frac12}
 e_h \|_{\Sigma}}_{=0} 
\\
&\lesssim \|h^{\frac12} \jump{\nabla e_h\cdot \bfn -\bfxi_h \cdot
  \bfn}\|_{\mathcal{F} \setminus\mathcal{F}_\Sigma} 
+ \|h^{\frac12}(\nabla
 e_h \cdot \bfn - \underbrace{\bfxi_h \cdot \bfn}_{=0})\|_{\Sigma} 
\\ 
&\lesssim \|\nabla e_h - \bfxi_h\|_\Omega.
\end{align*}
Then we proceed as in Corollary \ref{cor:H1apriori}.

\begin{prop}(Estimates of boundary data in natural norms.)\label{prop:boundaryest}
Assume that $(u,\bfp)$, $\bfp=A \nabla u$, is the solution to \eqref{weak_Cauchy} and
$(u_h,\bfp_h)$ the solution of \eqref{L2discmin}. Then the following
bound holds for the error in the approximation of the boundary data.
\begin{align*}
\|u - u_h\|_{H^{\frac12}(\Sigma)} + \|\psi - \bfp_h \cdot
\bfnu\|_{H^{-\frac12}(\Sigma)} 
&\lesssim
\|\delta \psi\|_{\Sigma}   + \|u - i_h u\|_{H^{\frac12}(\Sigma)} +
h^{\frac12} \|\psi - \psi_h\|_{\Sigma} 
\\
&\lesssim \|\delta
\psi\|_{\Sigma} + h^k |u|_{H^{k+1}(\Omega)}.
\end{align*}
\end{prop}
\begin{proof}
Since we have assumed that the Dirichlet data are unperturbed and we
have defined  $u_h\vert_{\Sigma}  = g_h = i_h u\vert_\Sigma$, it
follows using (\ref{eq:approx_Lagrange}) that
\[
\|u - u_h\|_{H^{\frac12}(\Sigma)}=\|u - i_h u\|_{H^{\frac12}(\Sigma)}
\lesssim \|u - i_h u\|_{H^{1}(\Omega)} \lesssim h^k |u|_{H^{k+1}(\Omega)} .
\]
Recalling that $\bfp_h \cdot
\bfnu\vert_\Sigma = \tilde \psi_h$ we may write, with $\psi_h$ the
$L^2$-projection of $\psi$ such that $\psi_h\vert_F = \piF \psi$,
    \begin{equation*}
\left<\psi - \tilde \psi_h, v\right>_{\Sigma}
= \left<\psi - \psi_h, v\right>_{\Sigma} +
\left<\psi_h - \tilde \psi_h, v\right>_{\Sigma}
\lesssim 
\left<\psi - \psi_h, v - v_h\right>_{\Sigma}
+ +|\left<\delta \psi,v_h\right>_{\Sigma}|.
    \end{equation*}
We now choose $v_h$ so that $v_h|_{F}=\piF v$. Using the stability of the $L^2$ projection, bounds for $v \in
H^1(\Sigma)$, interpolation, and the density of $H^1(\Sigma)$ in
$H^{\frac12}(\Sigma)$ we have that
for $v \in H^{\frac12}(\Sigma)$, with $ \|v\|_{H^{\frac12}(\Sigma)}=1$ there holds
\[
\| v - v_h \|_{\Sigma} \lesssim  h^{\frac12} \|v\|_{H^{\frac12}(\Sigma)}
\lesssim h^{\frac12}.
\]
After bounding the perturbation term using the Cauchy-Schwarz,
inequality, duality and the approximation of the $L^2$-projection
\[
|\left<\delta \psi,v_h-v\right>_{\Sigma}|  +|\left<\delta
  \psi,v\right>_{\Sigma}| \lesssim \|\delta
\psi\|_{H^{-\frac12}(\Sigma)} + h^{\frac12} \|\delta \psi\|_\Sigma,
\]
It follows form the above relations that for $\|v\|_{H^{\frac12}(\Sigma)}=1$,
\[
\left<\psi - \bfp_h \cdot
\bfnu, v  \right>_{\Sigma}\lesssim
h^{\frac12} \|\psi - \psi_h\|_{\Sigma} + \|\delta \psi\|_\Sigma
\]
Note that by definition of the $L^2$-projection, and recalling that $A$
is constant and $A\nabla i_h u\cdot \bfnu\vert_F \in \mathbb{P}_l(F)$,
\[
\|\psi - \psi_h\|_{\Sigma} \leq \|A\nabla u\cdot \bfnu - A\nabla i_h
u\cdot \bfnu\|_{\Sigma} \lesssim h^{k-\frac12} |u|_{H^{k+1}(\Omega)}.
\]
The last inequality followed by an application of \eqref{eq:trace} on all
the boundary faces in $\Sigma$ followed by
\eqref{eq:approx_Lagrange}. Combining the above bounds completes the proof.
\end{proof}

For the error analysis we must construct a function in $H^1(\Omega)$
such that both boundary conditions can be estimated simultaneously in
natural norms and which is close enough to $u_h$ in terms of the
residual terms estimated in Proposition \ref{prop:residualest}. For
the construction we follow the arguments of \cite{Bur17b}.

\begin{prop}\label{prop:pivot}
Let $(u_h,\bfp_h)$ be the solution of \eqref{eq:compact}. Then
 for some $h_0 > 0$, for all $h<h_0$ there exists $\tilde u_h$ such that $\tilde u_h\vert_\Sigma = u_h
\vert_\Sigma$ and
\[
\| A \nabla \tilde u_h - \bfp_h\|_{H^{-\frac12}(\Sigma)} + h^{-1}
\|\tilde u_h - u_h\|_\Omega + \|\nabla(\tilde u_h -
u_h)\|_\Omega\lesssim \|A \nabla u_h - \bfp_h\|_\Omega.
\]
\end{prop}

\begin{proof}
We decompose $\Sigma$ in disjoint, shape regular,
elements $\{\tilde F\}$, with diameter $O(h)$. To each element $\tilde F$
we associate a bulk patch $\tilde P$ that extends $O(h)$ into
$\Omega$, $\partial \tilde P \cap \Sigma = \tilde F$. On each patch we
will define a function $\varphi_{\tilde F} \in H^1_0(\tilde P)$ such
that
\[
\int_{\tilde F} A \nabla \varphi_{\tilde F} \cdot \bfnu ~\mbox{d}s =
\int_{\tilde F} ~\mbox{d}s =: \mbox{meas}_{d-1}(\tilde F)
\]
and
\begin{equation}\label{eq:vphistab}
h^{-1} \|\varphi_{\tilde F}\|_{\tilde P} + \|\nabla \varphi_{\tilde
  F}\|_{\tilde P} \lesssim h^{\frac{d}{2}}.
\end{equation}
Under the condition that $h$ is small enough we may take
$\varphi_{\tilde F} \in V_{h,0}^1$.
An example of construction of the $\{\varphi_{\tilde F} \}_{\tilde F}$
is given in appendix. We introduce
the projection on constant functions on $\tilde F$, $\pi_{\tilde F}
:L^2(\tilde F) \mapsto \mathbb{R}$ defined by $\pi_{\tilde F} v :=
\mbox{meas}_{d-1}(\tilde F)^{-1} \int_{\tilde F} v ~\mbox{d}s$. Then
consider $u_{\tilde F} := \pi_{\tilde F} (\bfp_h -A
\nabla u_h \cdot \bfnu )$ and define
\[
\tilde u_h := u_h + \sum_{\tilde F} u_{\tilde F} \varphi_{\tilde F}.
\]
It then follows by the definition of $\tilde u_h$ and an inverse
inequality that
\begin{multline}\label{eq:pert_approx1}
h^{-2} \|\tilde u_h - u_h\|^2_\Omega+\|\nabla(\tilde u_h -
u_h)\|^2_\Omega
\lesssim \sum_{\tilde F} u_{\tilde F}^2 h^{-2} \|\varphi_{\tilde F}\|^2_{\tilde
  P} \\
\lesssim \sum_{\tilde F} h^{1-d} h^d \|A \nabla u_h \cdot \bfnu -
\bfp_h\|_{\tilde F}^2 
\lesssim \|A \nabla u_h \cdot \bfnu -
\bfp_h\|_{\Omega}^2.
\end{multline}
For the second inequality we used \eqref{eq:vphistab} and $|u_{\tilde F}| \leq h^{\frac{1-d}{2}} \|A \nabla u_h \cdot \bfnu -
\bfp_h\|_{\tilde F}$ and for the third we applied the trace
inequality \eqref{eq:trace} to every element face in each $\tilde F$.
The bound of the flux on $\Sigma$ is shown observing that by the
definition of the $u_{\tilde F}$ and $\varphi_{\tilde F}$, for any $v
\in H^{\frac12}$ with $\|v\|_{H^{\frac12}(\Sigma)}=1$,
    \begin{equation*}
\left<A\nabla \tilde u_h - \bfp_h, v\right>_{\Sigma}
= \left<A\nabla \tilde u_h - \bfp_h, v - v_h\right>_{\Sigma}
\lesssim h^{\frac12} \|A\nabla \tilde u_h - \bfp_h\|_\Sigma.
    \end{equation*}
Here we $v_h$ is the piecewise constant function such that
$v_j\vert_{\tilde F} = \pi_{\tilde F} v$. We conclude by applying once
again the trace inequality \eqref{eq:trace} on every face in $\Sigma$
and using a triangle inequality and arguments similar to those used in
\eqref{eq:pert_approx1}, showing that
\[
h^{\frac12} \|A\nabla \tilde u_h - \bfp_h\|_\Sigma \lesssim \|A\nabla
u_h - \bfp_h\|_\Omega + \left(\sum_{\tilde F} u_{\tilde F}^2 h^{-2} \|\varphi\|^2_{\tilde
  P}\right)^{\frac12} \lesssim \|A\nabla
u_h - \bfp_h\|_\Omega
\]
\end{proof}

\begin{thm}\label{thm:condstab_est}(Conditional error estimates.)
Assume that $(u,\bfp)$ is the solution to \eqref{weak_Cauchy}, with $u
\in H^1(\Omega) \cap H^{k+1}(\Omega)$ and $\bfp = A \nabla u$,
$(u_h,\bfp_h)$ either the solution of \eqref{eq:compact}
(case $\zeta=1$) with
regularizing term given by \eqref{eq:tikhonov} and $k-1 \leq l \leq
k$, $m \ge l$, or the solution of
\eqref{eq:compact_red} (case $\zeta=0$).  Assume also that the
hypothesis of Theorems \ref{locbound} and \ref{globbound} are
satisfied and that $h< \min(h_0,\gamma_T^{-\frac1 {2k}})$, where $h_0$
is the bound from Proposition \ref{prop:pivot}. Then there holds for all $G$ as defined in
Theorem \ref{locbound}, for some $\tau \in (0,1)$,
\begin{equation}\label{eq:loc_error}
\|u - u_h\|_{G} \lesssim  
(1+\gamma_T^{\frac 1 2})^\tau
C_E\, h^{\tau k}
\end{equation}
where
\begin{equation}\label{eq:factor}
C_E \equiv C_E(u,\delta f, \delta \psi,h) \lesssim  \gamma_T^{-\frac12} (C_u  + C_{\bfp}
h^{l+1-k} + h^{-k} \|\delta
f\|_{\Omega} + h^{-\frac12-k+\frac{\zeta}{2}} \|\delta \psi\|_\Sigma)
\end{equation}
with $C_u$ and $C_{\bfp}$ defined in Proposition \ref{prop:residualest}. 
The following
global estimate also holds for some $\tau \in (0,1)$, 
\begin{equation}\label{eq:glob_error}
\|u-u_h\|_\Omega \lesssim C_E \frac{1}{|\log((1+\gamma_T^{\frac 1 2}) h^k)|^\tau}.
\end{equation}
\end{thm}
\begin{proof}
First observe that recalling the function $\tilde u_h$ from Proposition \ref{prop:pivot}
\begin{equation}\label{eq:triang_pivot}
\|u - u_h\|_{L^2(G)} \lesssim \|u - \tilde u_h\|_{L^2(G)} + h \|A\nabla
u_h - \bfp_h\|_{\Omega}.
\end{equation}
Since the second term is bounded in Proposition \ref{prop:residualest}
we only need to bound the first term of the right hand side.
To this end we will use that the error $\tilde e:= u-\tilde
u_h$ is a
solution to the equation \eqref{weak_Cauchy} for certain data. Observe
that using Propositions \ref{prop:boundaryest} and \ref{prop:pivot}
\begin{equation}\label{eq:boundary_est_pivot}
\|\tilde e\|_{H^{\frac12}(\Sigma)} + \|A \nabla \tilde e \cdot
\bfnu\|_{H^{-\frac12}(\Sigma)} \lesssim \|\delta
\psi\|_{\Sigma} + h^k |u|_{H^{k+1}(\Omega)}.
\end{equation}
Injecting $\tilde e$ in the weak formulation we see that for all $v
\in V_{\Sigma'}$,
\[
a(\tilde e,v) = a(u_h - \tilde u_h,v) + a(e,v) 
\]
where $e = u - u_h$.
Defining also the finite element residual, for all $v \in V_{\Sigma'}$, 
\begin{align} \nonumber
a(e,v) &= -(A \nabla e, \nabla v)_\Omega + (\mu e, v)_\Omega 
\\ \nonumber
&= (\bfxi -A \nabla e, \nabla v)_\Omega  
+  (\nabla \cdot \bfxi +  \mu e, v)_\Omega
-\left<\bfxi \cdot \bfnu,v \right>_{\Sigma} 
\\ \nonumber
&=: \left<r(e,\bfxi),v\right>_{(V_{\Sigma'})',V_{\Sigma'}}.
\end{align}
and comparing with equation \eqref{weak_Cauchy}, we may write the
right hand side
\begin{equation}\label{eq:error_rhs}
l(v) =  a(u_h - \tilde u_h,v) + \left<r(e,\bfxi),v\right>_{(V_{\Sigma'})',V_{\Sigma'}}.
\end{equation}
It remains to prove \eqref{bulk_data}. To this end we show the following bound 
\[
\|l(v)\|_{(V_{\Sigma'})'} \lesssim \tnorm{(e,\bfxi)}_{-\zeta} +\|z_h\|_{1,h}+ \|\bfxi \cdot \bfnu\|_{H^{-\frac12}(\Sigma)}+\|\delta f\|_{\Omega},
\]
where $\zeta = 1$ for the method \eqref{eq:compact} and $\zeta=0$ for
the method \eqref{eq:compact_red}.
Indeed, we can combine this bound with that of equation
\eqref{eq:boundary_est_pivot}, and apply the stability estimates in Theorems
\ref{locbound} and \ref{globbound} to the error $\tilde e$, resulting in the
claim.
First we use the Cauchy-Schwarz inequality on $a(u_h - \tilde u_h,v)$
followed by the bound of Proposition \ref{prop:pivot} leading to
\begin{equation*}
 a(u_h - \tilde u_h,v) \lesssim \|u_h - \tilde u_h\|_{H^1(\Omega)}
 \|v\|_{H^1(\Omega)} \lesssim \|A \nabla u_h - \bfp_h\|_{\Omega}
 \|v\|_{H^1(\Omega)} \leq \tnorm{(e,\bfxi)}_{-\zeta}\|v\|_{H^1(\Omega)}  .
\end{equation*}
When $\zeta=1$ and the method \eqref{eq:compact} is considered, we can
apply the orthogonality in the second term of the right hand side of \eqref{eq:error_rhs}. Choosing $v_h \in W^m$ to
be the elementwise $L^2$-projection of $v$ we obtain
\[
(\nabla \cdot \bfxi + \mu e, v)_\Omega = (\nabla \cdot \bfxi + \mu e, v -
v_h)_\Omega +s^*(z_h,v_h) - (\delta f, v_h)_\Omega.
\]
Using the approximation properties of the $L^2$-projection
it follows that
\[
(\nabla \cdot \bfxi + \mu e, v)_\Omega \lesssim (\|h (\nabla \cdot \bfxi +
ce)\|_\Omega+ \|\delta f\|_{\Omega}+\|z_h\|_{1,h}) \|v\|_{H^1(\Omega)}.
\]
Here we used also the fact that $s^*(z_h,v_h) \lesssim C\|z_h\|_{1,h} \|\nabla v_h\|_h
\lesssim \|z_h\|_{1,h} \|v\|_{H^1(\Omega)}$.

In case $\zeta=0$ the bound 
of the volume integral term is immediate by the
Cauchy-Schwarz inequality, 
\begin{align*}
(\bfxi-A \nabla e, \nabla v)_\Omega +  (\nabla \cdot \bfxi + \mu e, v)_\Omega  
&\leq
(\|A \nabla e - \bfxi\|^2_{\Omega} 
+\|\nabla \cdot \bfxi + \mu e\|^2_{\Omega})^{\frac12} \|v\|_{H^1(\Omega)}
\\
&\leq \tnorm{(e, \bfxi)}_0 \|v\|_{H^1(\Omega)}.
\end{align*}
For the boundary term we proceed using duality followed by the trace inequality
\[
\left<\bfxi \cdot \bfnu,v \right>_\Sigma \leq \|\bfxi \cdot
\bfnu\|_{H^{-\frac12}(\Sigma)} \|v\|_{H^{\frac12}(\Sigma)} \lesssim \|\bfxi \cdot
\bfnu\|_{H^{-\frac12}(\Sigma)}\|v\|_{H^1(\Omega)}.
\]
Collecting these bounds we obtain, with the two cases distinguished by $\zeta$,
\begin{multline*}
-(A \nabla e
- \bfxi, \nabla v)_\Omega +  (\nabla \cdot \bfxi + \mu e, v)_\Omega  - \left<\bfxi \cdot \bfnu,v \right>_{\Sigma}\\
\lesssim (\tnorm{(e,\bfxi)}_{-\zeta}+\zeta \|z_h\|_{1,h}+ \zeta\|\delta f\|_{\Omega}+ \|\bfxi \cdot \bfnu\|_{H^{-\frac12}(\Sigma)}) \|v\|_{H^1(\Omega)}.
\end{multline*}
We conclude that by Propositions \ref{prop:residualest},
\ref{prop:boundaryest} and \ref{prop:pivot} there holds
\begin{equation}\label{eq:rhs_bound}
\begin{array}{ll}
\|l(v)\|_{(V_{\Sigma'})'} & \lesssim \|u_h - \tilde u_h\|_{H^1(\Omega)}+\| r(e,\bfxi)\|_{(V_{\Sigma'})'} \\
& \lesssim \tnorm{(e,\bfxi)}_{-\zeta}  +\zeta\|z_h\|_{1,h}
+ \zeta \|\delta f\|_{\Omega}+ \|\bfxi \cdot \bfnu\|_{H^{-\frac12}(\Sigma)} \\
& \lesssim  C_u
h^{k}+C_{\bfp} h^{l+1} + \|\delta f\|_\Omega + h^{-\frac12+\frac{\zeta}{2}} \|\delta \psi\|_\Sigma.
\end{array}
\end{equation}
Here (and below) we use that $h^{\frac12} \|\psi - \psi_h\|_{\Sigma} \lesssim h^k
|u|_{H^{k+1}(\Omega)} \lesssim C_u
h^{k}$ to absorb the boundary error contribution.
We are now in position to prove the error estimate using Theorems
\ref{locbound} and \ref{globbound}.
To simplify the notation, we write $C_E = C_E(u,\delta f, \delta \psi,h)$.
First note that the error $\tilde e$ is a solution to the problem
\eqref{weak_Cauchy} with the right hand side defined by
\eqref{eq:error_rhs}. By \eqref{eq:boundary_est_pivot}
the inequality
\eqref{bc_data} is satisfied with 
\[
\eta \lesssim \|\delta \psi\|_{\Sigma} + h^k \|u\|_{H^{k+1}(\Omega)}.
\]
By \eqref{eq:rhs_bound} the inequality \eqref{bulk_data} holds with 
\[
\varepsilon \lesssim \gamma_T^{\frac 1 2}C_E h^k.
\]
The a priori bounds \eqref{eq:aprioriloc} and \eqref{eq:aprioriglob}
follow from Corollary \ref{cor:H1apriori} with
\[
E_0 \leq E \lesssim C_E.
\]
We then observe that, assuming $h < \gamma_T^{-\frac1 {2k}}$,
\[
E_0+\varepsilon+\eta \lesssim C_E,
\quad \varepsilon + \eta \lesssim (1+\gamma_T^{\frac 1 2}) C_E h^k.
\]
Applying these bounds in \eqref{eq:locbound} we obtain a 
bound for the first term on the right hand side of
\eqref{eq:triang_pivot},
\[
\|u - \tilde u_h\|_{L^2(G)} \lesssim (1+\gamma_T^{\frac 1 2})^\tau
C_E \, h^{\tau k}
\]
 leading to
the local error bound \eqref{eq:loc_error}. The global error bound
\eqref{eq:glob_error} is
obtained by inserting the above bounds on $E$,
$\varepsilon$ and $\eta$ into \eqref{eq:globbound}.
\end{proof}
\begin{rem}
Observe that from the definition of $C_E$ it follows that the bound
makes sense only when $h^{-k} \|\delta
f\|_{\Omega} + h^{-\frac12-k+\frac{\zeta}{2}} \|\delta \psi\|_\Sigma$
is small compared to $|u|_{H^{k+1}(\Omega)}$.
\end{rem}
\begin{rem}
It is possible to derive corresponding local estimates in the
$H^1$-norm, if similar stability estimates are available. Such
estimates will have the same rate as those in the $L^2$-norm, which is
expected to be sharp since no adjoint argument is available to improve
the convergence in the $L^2$-norm. In the numerical section we will
see that depending on the geometry of the Cauchy problem, the $L^2$-norm can perform better
than the $H^1$-norm.
\end{rem}
\subsection{Interlude on the well-posed case}
In case the problem under study satisfies the assumptions of the
Lax-Milgram's Lemma, Theorem \ref{thm:condstab_est} immediately leads to
optimal error estimates in the $H^1$-norm, using the stability
estimate $\|u\|_{H^1(\Omega)}\lesssim \|l\|_{H^{-1}(\Omega)}$ instead of the conditional stability. Since the problem is
well-posed we may take $\gamma_T=0$. Here we will instead focus on the convergence in the $L^2$-norm. In
this case we only require that the adjoint problem is well-posed and
satisfies a shift theorem for the $H^2$ semi-norm.  This analysis,
which is an equivalent of that in \cite{Bu13}, for the mixed finite
element case,
includes indefinite elliptic problems as well. For
simplicity we here restrict the discussion to the case of a convex
polygonal domain $\Omega$, 
homogeneous Dirichlet conditions and $A$ the identity matrix. 
We consider the problem: find $u \in H^1(\Omega)$ such that
\begin{equation}\label{eq:primal_wp}
\Delta u + \mu u = f \mbox{ in } \Omega
\end{equation}
with $u\vert_{\partial \Omega} = 0$ and the associated adjoint problem
\begin{equation}\label{adjoint}
\Delta \varphi + \mu \varphi = \psi \mbox{ in } \Omega
\end{equation}
with $\varphi\vert_{\partial \Omega} = 0$. The weak formulation of the
forward problem takes the form: find $u\in H_0^1(\Omega)$ such that 
\[
a(u,v) = (f,v)_{\Omega}\quad \forall v \in H^1_0(\Omega).
\]
We assume that $\mu$ is
such that the problem admits a unique solution, using Fredholm's alternative. Under the assumptions the adjoint
problem satisfies the regularity bound
\begin{equation}\label{eq:adjoint_stab}
\|\varphi\|_{H^2(\Omega)} \lesssim \|\psi\|_\Omega.
\end{equation}
For the discretization we here consider only
the formulation 
\eqref{eq:EL_1}-\eqref{eq:EL_2} with stabilizers given by
\eqref{eq:primal_stab} and \eqref{eq:tikhonov}, $\gamma_T=0$, but the result also holds
for the stabilizers \eqref{eq:primal_stab} and \eqref{eq:L2LS}. Homogeneous
Dirichlet conditions are imposed on all of $\partial \Omega$ in the
space $V_0^k$ and no boundary conditions are imposed on the space
where the fluxes are sought, $D^l=RT_{l}$. 
We now prove optimal convergence in the $L^2$-norm.
\begin{proposition}
Let $u \in H^{k+1}(\Omega) \cap H^1_0(\Omega)$ by the solution of
\eqref{eq:primal_wp} and $(u_h, \bfp_h, z_h) \in V_0^k \times D^l
\times W^m$, with $k-1\leq l \leq k$ and $l \leq m \leq k$, be a
solution of \eqref{eq:EL_1}-\eqref{eq:EL_2} with stabilizers given by
\eqref{eq:primal_stab} and \eqref{eq:tikhonov}. Then there holds
\[
\|u - u_h\|_{L^2(\Omega)} \lesssim h^{k+1} |u|_{H^{k+1}(\Omega)}.
\]
\end{proposition}
\begin{proof}
First observe that the result of Proposition \ref{prop:residualest}
holds with $\zeta=1$.
To prove optimal convergence of the error in the $L^2$-norm we first
show that the $L^2$-norm of the Lagrange multiplier goes to zero with
$O(h^{k+1})$. To this end let $\psi=z_h$ in \eqref{adjoint}, then
\[
\|z_h\|_\Omega^2 = (z_h, \Delta \varphi + \mu \varphi)_\Omega
= (z_h, \Delta \varphi - \nabla \cdot \bfR_h (\nabla \varphi) + \mu (
 \varphi - i_h \varphi))_\Omega - s[(u_h,\bfp_h),( i_h
 \varphi, \bfR_h (\nabla \varphi))]
\]
For the last equality we used equation \eqref{eq:EL_1} with $\bfq_h =
\bfR_h (\nabla \varphi)$ and $v_h = i_h \varphi$.
Using the divergence theorem elementwise we have
\begin{align*}
 (z_h, \Delta \varphi - \nabla \cdot \bfR_h (\nabla \varphi))_\Omega
& \leq \|z_h\|_{1,h} (\|\nabla \varphi - \bfR_h (\nabla
 \varphi)\|_\Omega + \|h^{\frac12} (\nabla \varphi - \bfR_h (\nabla
 \varphi)) \cdot \bfn\|_{\mathcal{F}}) 
 \\
& \lesssim \|z_h\|_{1,h} h |\varphi|_{H^2(\Omega)}
\end{align*}
where we recall that for this case $\|h^{-\frac12} z_h\|_{\partial
  \Omega} \leq \|z_h\|_{1,h}$ since no boundary conditions are imposed on the
flux variable. For the second term there holds using the
Cauchy-Schwarz inequality and the approximation properties of the
nodal and
Raviart-Thomas interpolants \eqref{eq:approx_Lagrange}, \eqref{eq:approx_RT}
\begin{align*}
(\nabla u_h - \bfp_h, \nabla i_h
 \varphi -  \bfR_h (\nabla \varphi))_\Omega
 &\lesssim \|\nabla u_h -
 \bfp_h\|_\Omega (\|\nabla i_h
 \varphi - \nabla \varphi\|_\Omega+\|\bfR_h (\nabla \varphi) - \nabla
 \varphi\|_\Omega)
 \\
&\lesssim  \|\nabla u_h -
 \bfp_h\|_\Omega h |\varphi|_{H^2(\Omega)}
\end{align*}
and
\[
t(i_h \varphi,u_h) \lesssim 
t(u_h,u_h)^{\frac12}\mu h^2 \|\varphi\|_{H^2(\Omega)}
\]
Using the stability of the adjoint problem we conclude that
\[
\|z_h\|_\Omega \lesssim h (\|\nabla u_h -
 \bfp_h\|_\Omega+\|z_h\|_{1,h}+ t(u_h,u_h)^{\frac12}) \lesssim h^{k+1} |u|_{H^{k+1}(\Omega)}.
\]
where we used the result of Proposition \ref{prop:residualest}, for unperturbed data
in the second inequality. We now proceed to prove the $L^2$-error
estimate. Let $e = u-u_h$, $\bfxi = \nabla u - \bfp_h$ and $\psi = e$ in \eqref{adjoint}, then
\begin{equation}\label{eq:error_rep}
\|e\|_\Omega^2 = (\nabla \varphi, \nabla e)_{\Omega} + (\mu \varphi,
e)_\Omega = 
-(\nabla \varphi, \nabla u_h - \bfp_h)_{\Omega} + (\nabla \cdot \bfxi + \mu e,
\varphi)_\Omega.
\end{equation}
We now proceed by applying Galerkin orthogonality in each of the two
terms on the right hand side. For the first term we get
\begin{align*}
(\nabla \varphi, \nabla u_h - \bfp_h)_{\Omega} 
&= ((1 - \bfR_h) \nabla
\varphi , \nabla u_h - \bfp_h)_{\Omega} - (\nabla \cdot \bfR_h (\nabla
\varphi), z_h) 
\\
&\leq h |\varphi|_{H^2(\Omega)} \|\nabla u_h -
\bfp_h\|_{\Omega} + |\varphi|_{H^2(\Omega)} \|z_h\|_\Omega.
\end{align*}
Here we used that $\nabla \cdot \bfR_h (\nabla
\varphi) = \pi_{X,l-1} \Delta \varphi$, and then the stability of the $L^2$-projection.
In the second term of the right hand side of \eqref{eq:error_rep} we use \eqref{eq:EL_2} leading to
\begin{align*}
(\nabla \cdot \bfxi + \mu e, \varphi)_\Omega 
&= (\nabla \cdot \bfxi + \mu e,
\varphi - i_h \varphi)_\Omega + s^*(z_h, i_h \varphi)
\\
&\lesssim h (\|h
(\nabla \cdot \bfxi + \mu e)\|_\Omega+ \|z_h\|_{1,h})  |\varphi|_{H^2(\Omega)}.
\end{align*}
We conclude by applying \eqref{eq:adjoint_stab} in \eqref{eq:error_rep} together
with Proposition \ref{prop:residualest} to obtain
\[
\|e\|_\Omega \lesssim h (\|h (\nabla \cdot \bfxi + \mu e)\|_\Omega +  \|\nabla u_h -
\bfp_h\|_{\Omega}+\|z_h\|_{1,h}) + \|z_h\|_\Omega \lesssim h^{k+1} |u|_{H^{k+1}(\Omega)}.
\]
\end{proof}
\section{Iterative solution of the inf-sup stable system}\label{sec:iterative}
Clearly the elimination of the dual variable is an important gain
compared to the original constrained system, in particular since the
resulting system is symmetric, positive definite and therefore can be
solved using the conjugate gradient method. We will here show how the
reduced method can be used to solve the full system in an iterative
procedure, which allows to recover the conservation properties and
error estimates of the full system while only solving the linear
system associated to the reduced system. The idea is to use the
Euler-Lagrange equations with the dual stabilizer
\eqref{eq:L2LS}, which leads to the mixed least squares method,
but consider the dual stabilizer as a perturbation that is eliminated
through iteration. The iterative scheme takes the form: let $z^0_h=0$
compute for $\kappa=0,1,2,3,4\hdots$
\begin{align}\label{eq:EL_ite1}
s[(u_h^{\kappa+1},\bfp^{\kappa+1}_h),(v_h,\bfq_h)]+b(\bfq_h, v_h,z^{\kappa+1}_h) & = 0 \\
b(\bfp^{\kappa+1}_h, u^{\kappa+1}_h,w_h) - s^*(z^{\kappa+1}_h,w_h)& = 
(f,w_h)_\Omega - s^*(z^{\kappa}_h,w_h),\label{eq:EL_ite2}
\end{align}
where $s$ and $s^*$ is defined by \eqref{eq:primal_stab} and \eqref{eq:L2LS}.
Clearly if the iteration converges, the resulting discrete solution
solves the inf-sup stable formulation for which $s^*\equiv 0$.
We will now prove the convergence of the scheme.
\begin{prop}\label{prop:iter}
Assume that $\gamma_T>0$.
Let $\kappa \rightarrow \infty$ in \eqref{eq:EL_ite1}-\eqref{eq:EL_ite2},
then $(u_h^{\kappa},\bfp^{\kappa}_h,z_h^\kappa) \rightarrow (u_h,\bfp_h,z_h)$
solution of \eqref{eq:EL_1}-\eqref{eq:EL_2}, with $s^*\equiv 0$.
\end{prop}
\begin{proof}
By linearity it is enough to prove that $(u_h^{\kappa},\bfp^{\kappa}_h,z_h^\kappa)$
goes to zero if $f\equiv 0,\, g\equiv 0,\, \psi \equiv 0$ in
\eqref{eq:EL_ite1}-\eqref{eq:EL_ite2}.
By taking $v_h = u_h^{\kappa+1}$, $\bfq_h= \bfp^{\kappa+1}_h$ and $w_h = -
z_h^{\kappa+1}$ and summing over $\kappa \in 0,\hdots n-1$ we obtain
\[
\sum_{\kappa=0}^{n-1}\left(
  s[(u_h^{\kappa+1},\bfp^{\kappa+1}_h),(u_h^{\kappa+1},\bfp^{\kappa+1}_h)] +
  s^*(z^{\kappa+1}_h-z^{k}_h,z^{\kappa+1}_h) \right)= 0
\]
and therefore using the telescoping sum
\[
\frac12 \|z^{n}_h\|_\Omega^2 + \sum_{\kappa=0}^{n-1}\left(
  s[(u_h^{\kappa+1},\bfp^{\kappa+1}_h),(u_h^{\kappa+1},\bfp^{\kappa+1}_h)] +  \frac12
  \|z^{\kappa+1}_h-z^{\kappa}_h\|_\Omega^2\right) = \frac12 \|z^{0}_h\|_\Omega^2.
\]
It follows that 
\[
\|A \nabla u_h^{\kappa} - \bfp_h^\kappa\|_{\Omega} +
\gamma_T\|h^k \nabla u_h^\kappa\|_\Omega+\|z^{\kappa+1}_h-z^{\kappa}_h\|_\Omega\rightarrow 0 \mbox{ when } k
\rightarrow \infty.
\]
Observe that for $\gamma_T>0$ this implies (by Poincar\'e's
inequality) that $u_h = \lim_{\kappa \rightarrow 0} u_h^\kappa= 0$ and
$\bfp_h = \lim_{\kappa \rightarrow 0} \bfp_h^\kappa = {\bf{0}}$. Using Theorem
\ref{infsup} we then conclude that $z^\kappa \rightarrow 0$. 
\end{proof}
\begin{rem}
If $k=1$ and $A$ is the identity, the conclusion of Proposition \ref{prop:iter} holds also for
$\gamma_T=0$. To see this recall the discussion after Remark \ref{rem:low_order} implying that 
\[
\|h \nabla u_h^\kappa\|_\Omega \lesssim \|\nabla u_h^{\kappa} - \bfp_h^\kappa\|_{\Omega}.
\]
The consequence is that $u_h = \lim_{\kappa \rightarrow 0} u_h^\kappa= 0$ as before.
\end{rem}
\section{Numerical example}
As a numerical illustration of the theory we consider the original Cauchy
problem discussed by Hadamard. In \eqref{Cauchy_strong} let $A=I$,
$\mu=0$, $f=0$, $\Omega:= (0,\pi)
\times (0,1)$, $\Gamma_N:=\{ x \in (0,\pi); y=0\}$, $\Gamma_D:=
\Gamma_N\cup \{x \in \{0,\pi \}; y \in (0,1)  \}$ and 
\begin{equation}\label{Cauchy_data}
\psi := -b_n \sin(n x).
\end{equation}
It is then straightforward to verify that 
\begin{equation*}
u_n = b_n n^{-1} \sin(nx) \sinh(n y)
\end{equation*}
solves \eqref{Cauchy_strong}. An example of the
exact solution for $n=5$ is given in Figure \ref{exact}. One may easily show show that the
choice $b_n = n^{-p}$, $p > 0$ leads to $\psi \rightarrow 0$ uniformly
as $n \rightarrow \infty$,
whereas, for any $y>0$, $u_n(x,y)$ blows up. Stability can only be obtained
conditionally, under the assumption that $\|u_n\|_{H^1(\Omega)} < E$ for some $E>0$, leading to the
relations \eqref{eq:locbound} and \eqref{eq:globbound}.

We choose $b_n:=1$ in \eqref{Cauchy_data} and impose homogenoeous
Dirichlet condition $u=0$ on $x \in (0,\pi),\, y=0$. On the lateral
boundaries $x=0$, $x=\pi$ we either impose homogenoeous
Dirichlet condition (case 1) or impose nothing (case 2). With these
data we then solve the
resulting Cauchy problem \eqref{Cauchy_strong}. We study the error in the
relative $L^2$-norms,
\begin{equation*}
\frac{\|u - u_h\|_{\Omega_\sigma}}{\|u\|_{\Omega_\sigma}}, \mbox{ where } \Omega_\sigma := (0,\pi)\times (0,\sigma),\quad \sigma \in \{1/2,\, 1\}.
\end{equation*}
We will also consider the relative $H^1$-semi-norm defined similarly.
In the graphics below, errors in the $L^2$-norm will be marked with circle
markers '$\circ$' and the error in the relative $H^1$-semi norm with
square markers '{\tiny{$\square$}}'. The case $\sigma=1$ will be
indicated with a filled markes, whereas the one for $\sigma=1/2$ with not filled.
%
All computations below were performed using
formulation \eqref{eq:compact} in the
package FreeFEM++ \cite{He12}. We consider the cases $k=1$ and $k=2$
for increasingly oscillating data with $n=1$ and $n=5$. To set
the regularization parameter $\gamma_T$ we performed a series of
computations on a mesh with $240 \times 80$ elements and unperturbed
data. We then chose the first $\gamma_T$ for which the influence of
the regularizing term was visible in the form of increasing error. The
resulting parameter was $10^{-4}$ both for $k=1$ and $k=2$. Observe
that for unperturbed data the regularization parameter could be chosen
to be zero. To minimize the influence of mesh structure we used
Union-Jack meshes, an example is given in Figure \ref{mesh}. We used
the iterative method of Section \ref{sec:iterative} to solve the
linear system and obtained convergence to $10^{-6}$ on the $L^2$-norm
of the increment after less than three iterations in all cases. In
experiments not presented here we used the reduced method and observed
similar accuracy of the approximations as those reported below.
\subsection{Case 1}
 In Figure \ref{k1} we show computations performed on a sequence
of structured meshes using $k=1$. From left to right we have $n=1$, and
$n=5$. We see that when $n=1$ the $H^1$ and $L^2$ errors converge with the
optimal orders $O(h^{k})$ and $O(h^{k+1})$ respectively.  For $n=5$ on
the other hand, (right plot of Figures \ref{k1} and \ref{k2}), when $k=1$ the relative error remains above $30\%$ on
all the meshes and the solution is clearly not resolved.

Increasing the order to $k=2$ changes the behaviour dramatically. The
results for this case is reported in Figure \ref{k2}. Here the dotted
reference lines are if $y \sim h^2$ and $y \sim h^3$ and we observe
that the $H^1$ and $L^2$ errors have optimal convergence, but the
local and the global errors.
For $n=5$ on the
coarse scales the problem is completely underresolved also in this
case. however the performance for $n=5$ is very different when $k=1$
and $k=2$. indeed for $k=1$ we do not observe any convergence on the
considered scales, whereas for $k=2$, the error decreases as expected
from the second refinement. Indeed we observe $O(h^2)$ convergence
from $O(1)$ errors to an error of order $10^{-4}$ for
$h=0.01$. Clearly there is a strong ``pollution'' effect of the Cauchy problem due
to oscillation in data. It appears that, similarly as for Helmholtz equation, using higher order
approximation leads to a method that is more robust in handling this phenomenon.

In Figure \ref{n3pert} we consider a similar computation with $n=3$,
but this time using the data $\tilde \psi = (1 + \delta u_{rand})
\psi$ where $u_{rand}$ is a finite element function where each degree
of freedom has been set randomly to a value in $[0,1]$. In the left
plot we give the convergence for the case $k=1$ and $\delta=0.02$ and in the right
$k=2$ with the same level of the perturbation. In both cases we observe stagnation when the error is of the
size of the perturbation. On finer meshes we also note that the error can
grow under refinement, this is consistent with theory (recall the
inverse power of $h$ in $C_E$ of \eqref{eq:factor}.)

\subsection{Case 2}
For the second case we only consider $n=1$. In Figure \ref{case2} we
report the relative errors for the cases $k=1$ (left plot) and $k=2$
(right plot). For $k=1$ the errors decrease during refinement. And on
the finest mesh, $h=1/400$, the error is $O(10^{-2})$ for the $L^2$-errors
and the local $H^1$-error. The global $H^1$-error is approximately a
factor four larger. All error quantities have similar behavior under
refinement. For $k=2$ on the other hand the errors are a factor $10$
smaller for comparable mesh-sizes and convergence appear to be
logarithmic for all quantities, which coincides with theory since all
domains where the errors are measured reach the undefined boundary.
For $h<0.02$ the errors grow in this case. For computations on finer
meshes not reported here this growth continuous. This saturation and error
growth in the case $k=2$ is most likely due to finite precision.
\begin{figure}[hbt]
\begin{center}
\includegraphics[origin=c,width=5in]{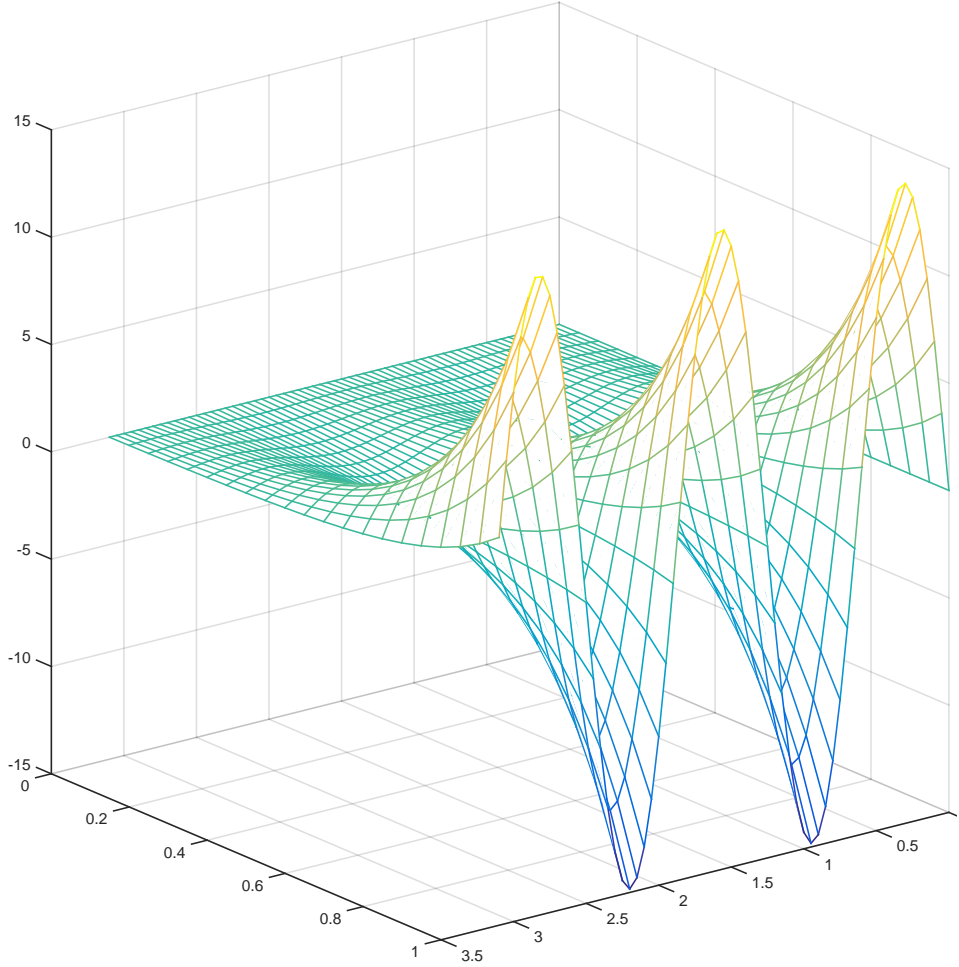}
\end{center}
\caption{Carpet plot of exact solution for $n=5$}\label{exact}
\end{figure}
\begin{figure}[hbt]
\begin{center}
\vspace{-2cm}
\includegraphics[angle=90,origin=c,width=5in]{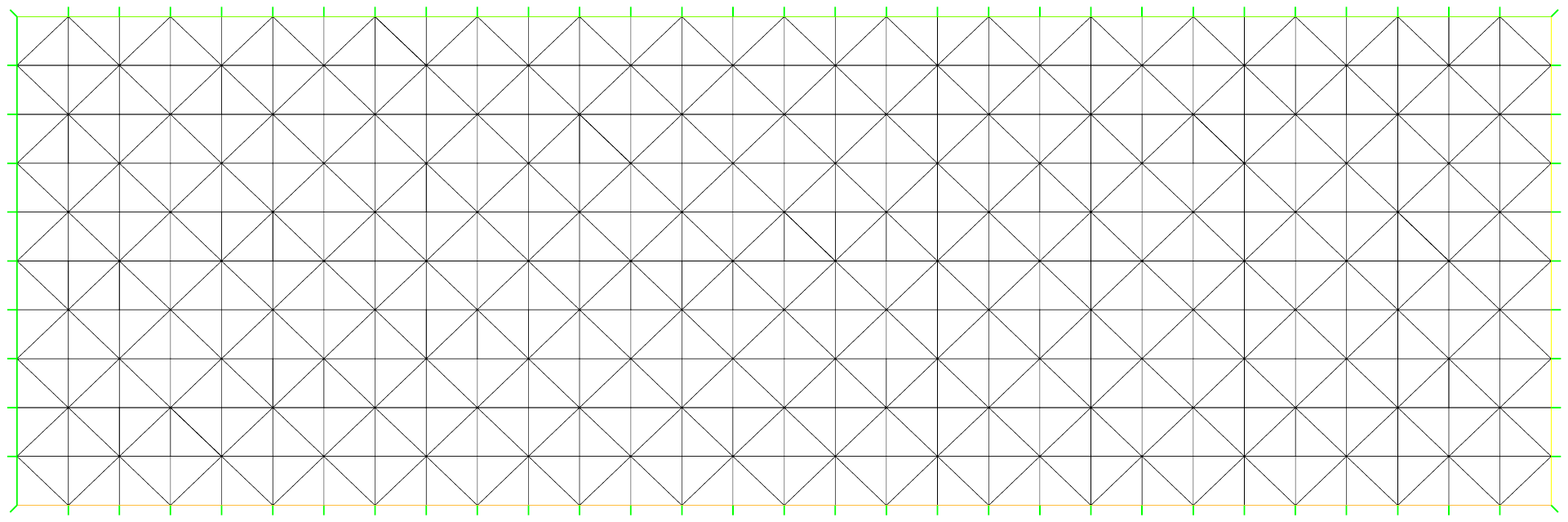}
\vspace{-4cm}
\end{center}
\caption{Example of a computational mesh}\label{mesh}
\end{figure}
\begin{figure}[hbt]
\begin{center}
\includegraphics[width=2.5in]{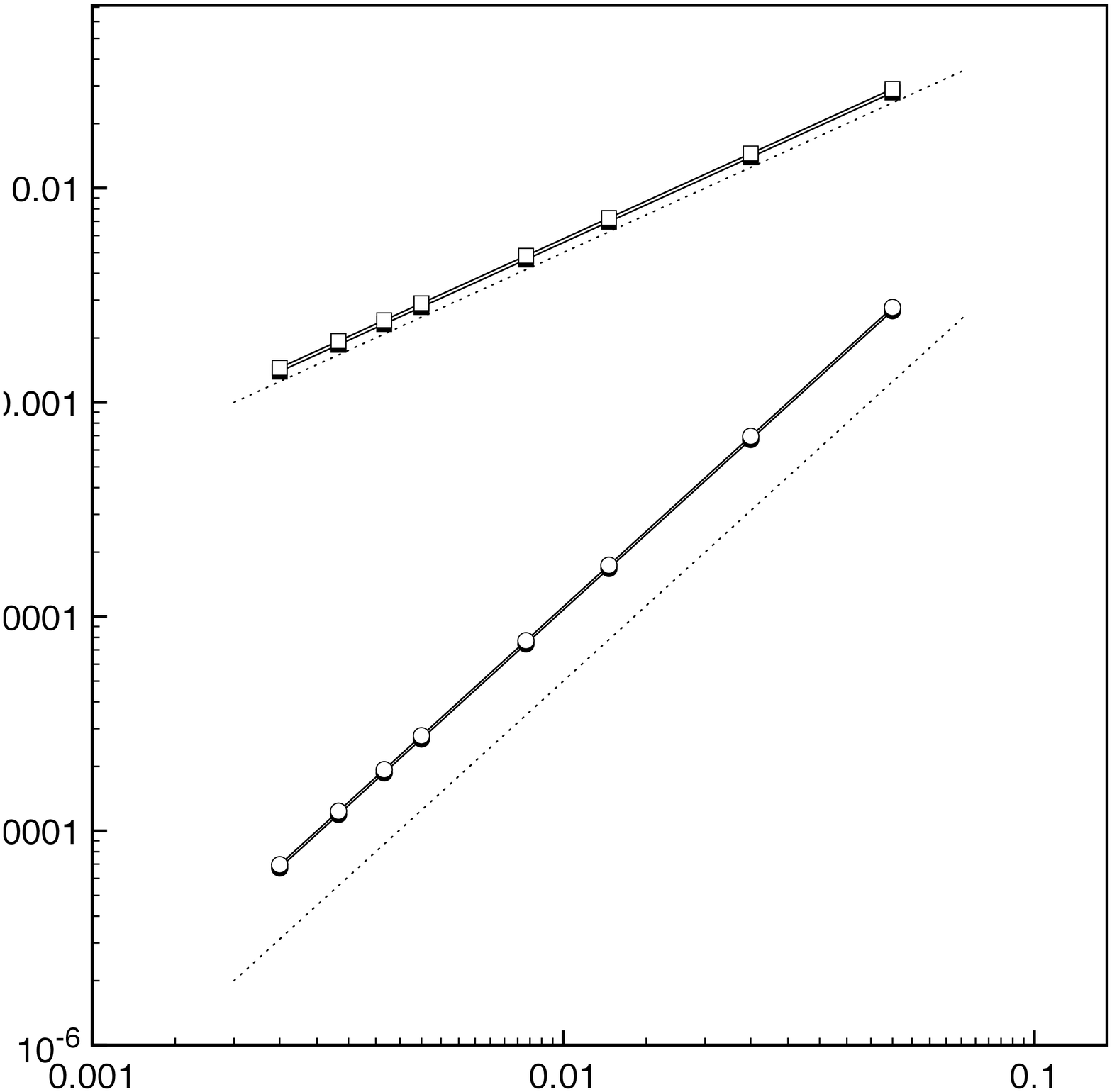}
\includegraphics[width=2.5in]{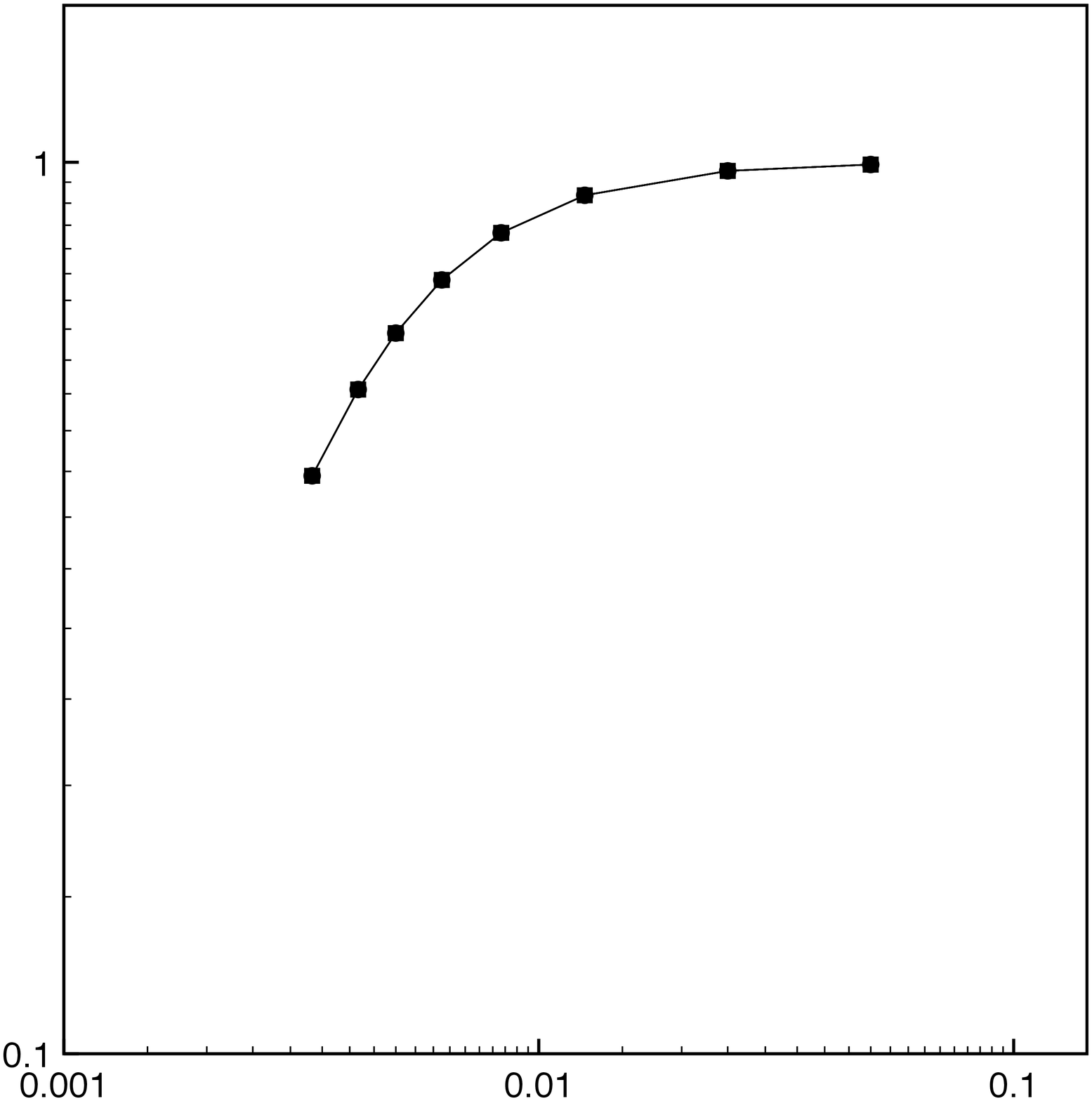}
\end{center}
\caption{Relative error against mesh size. $k=1$, from left to right:
  $n=1$, $n=5$. Square markers indicate 
  $H^1$-seminorm errors, circle markers indicate $L^2$-errors, filled
  markers indicate global errors and not filled markers indicate
  local errors. Dotted reference lines (left plot): top $y=O(h)$, bottom $y=O(h^2)$}\label{k1}
\end{figure}
\begin{figure}[hbt]
\begin{center}
\includegraphics[width=2.5in]{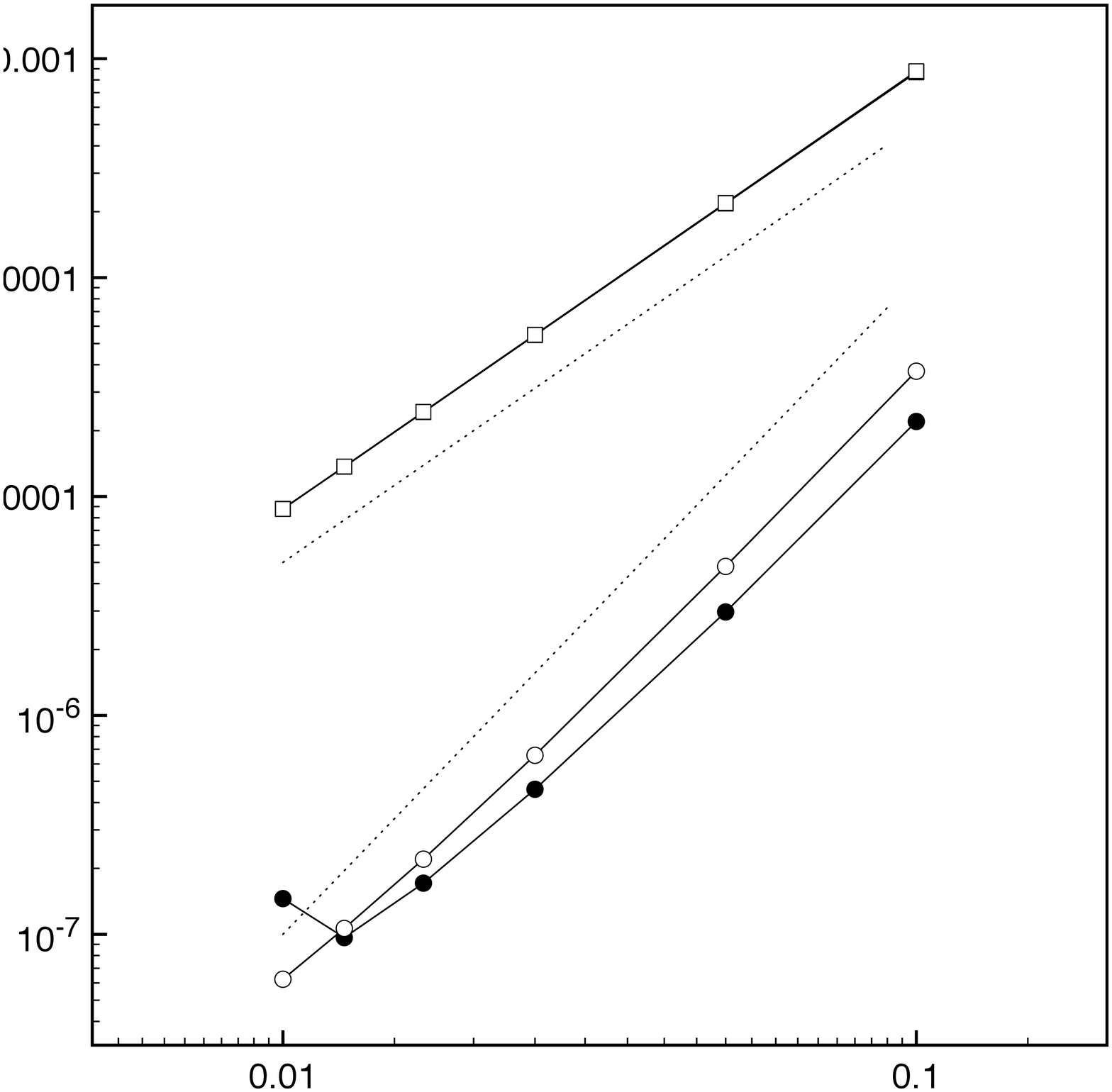}
\includegraphics[width=2.5in]{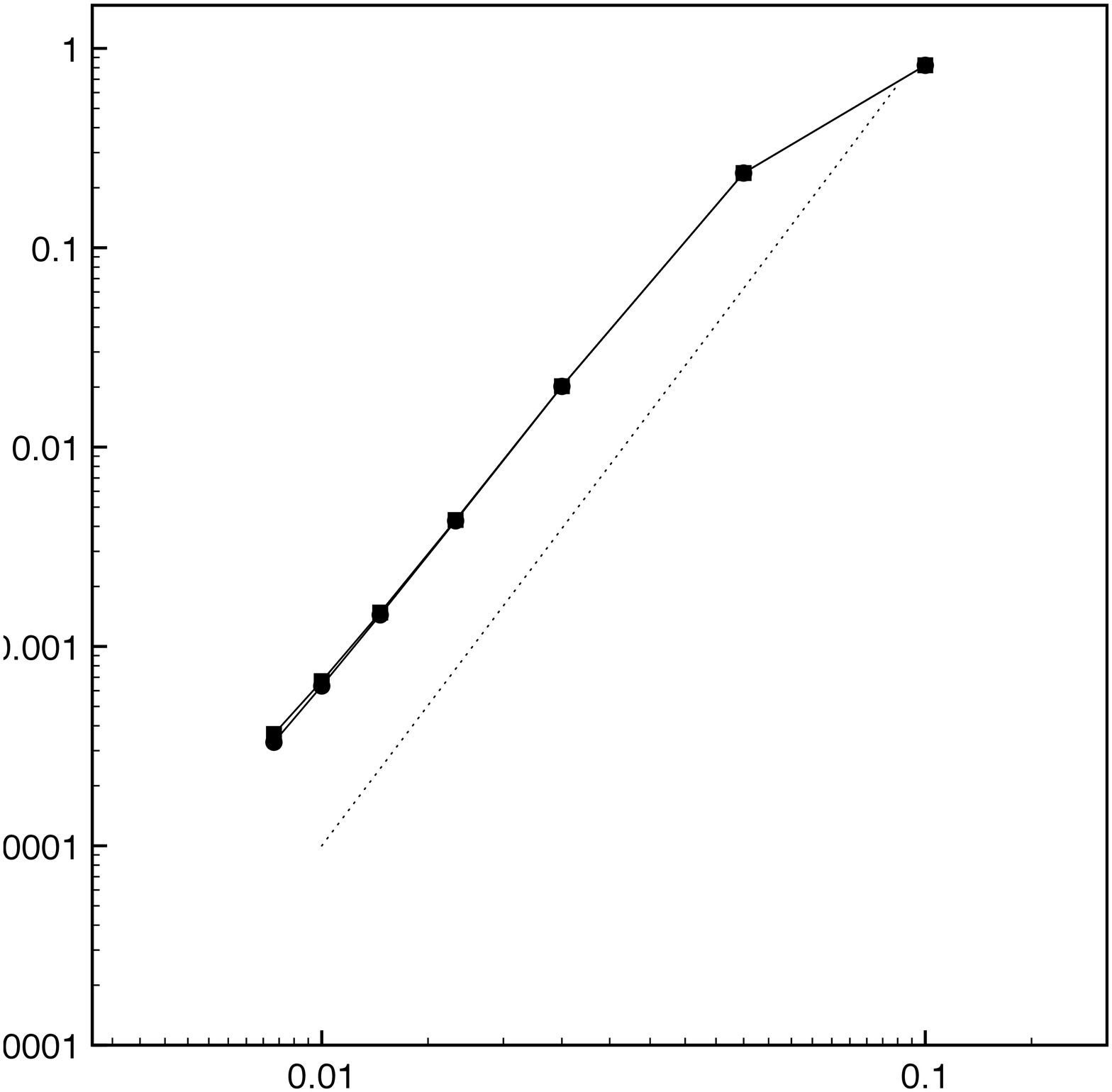}
\end{center}
\caption{Relative error against mesh size. $k=2$, from left to right:
  $n=1$, $n=5$. Square markers indicate 
  $H^1$-seminorm errors, circle markers indicate $L^2$-errors, filled
  markers indicate global errors and not filled markers indicate
  local errors. Dotted reference lines (left plot): top $y=O(h^2)$,
  bottom $y=O(h^3)$, (right plot) $y=O(h^2)$}\label{k2}
\end{figure}
\begin{figure}[hbt]
\begin{center}
\includegraphics[width=2.5in]{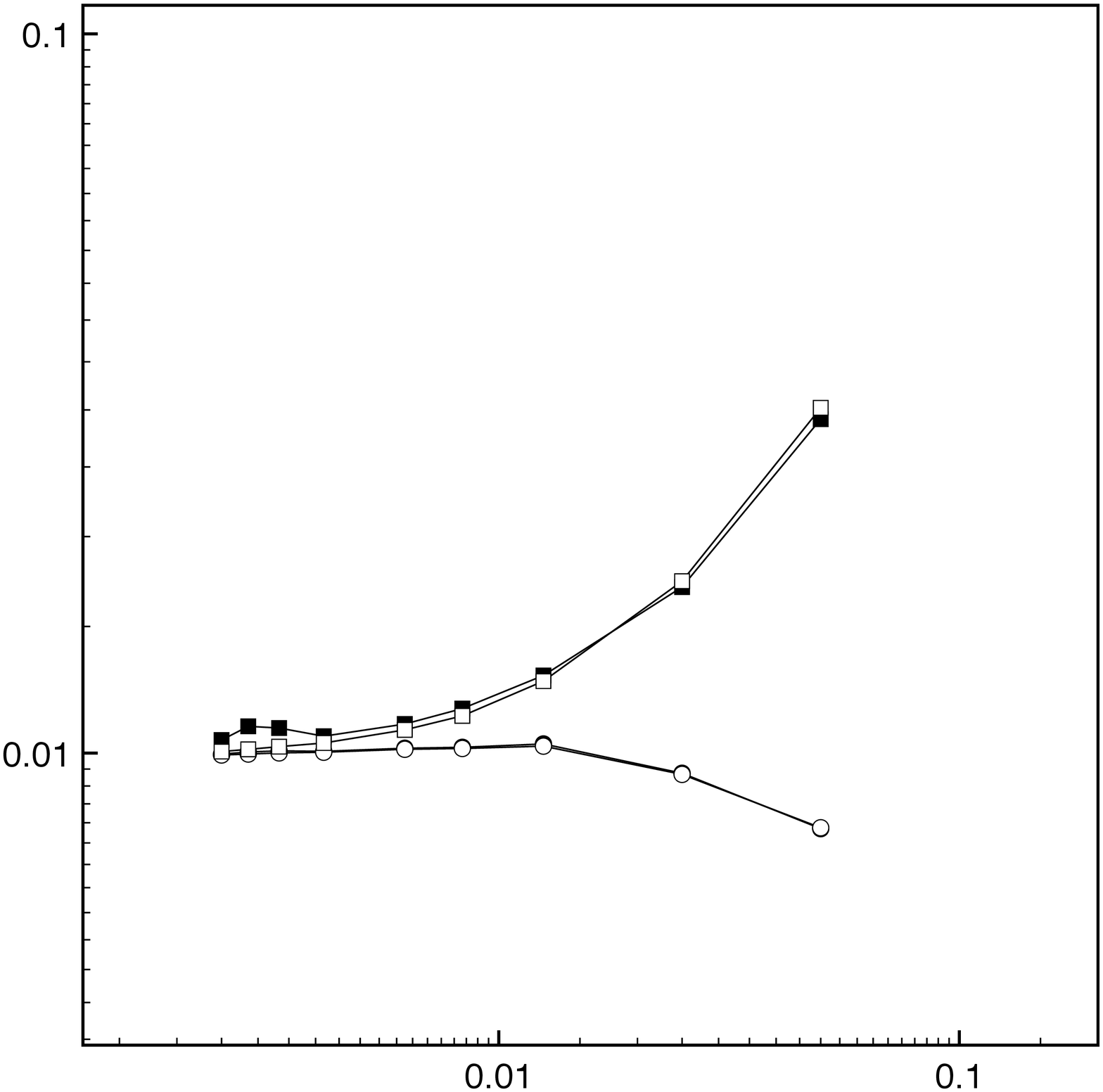}\includegraphics[width=2.5in]{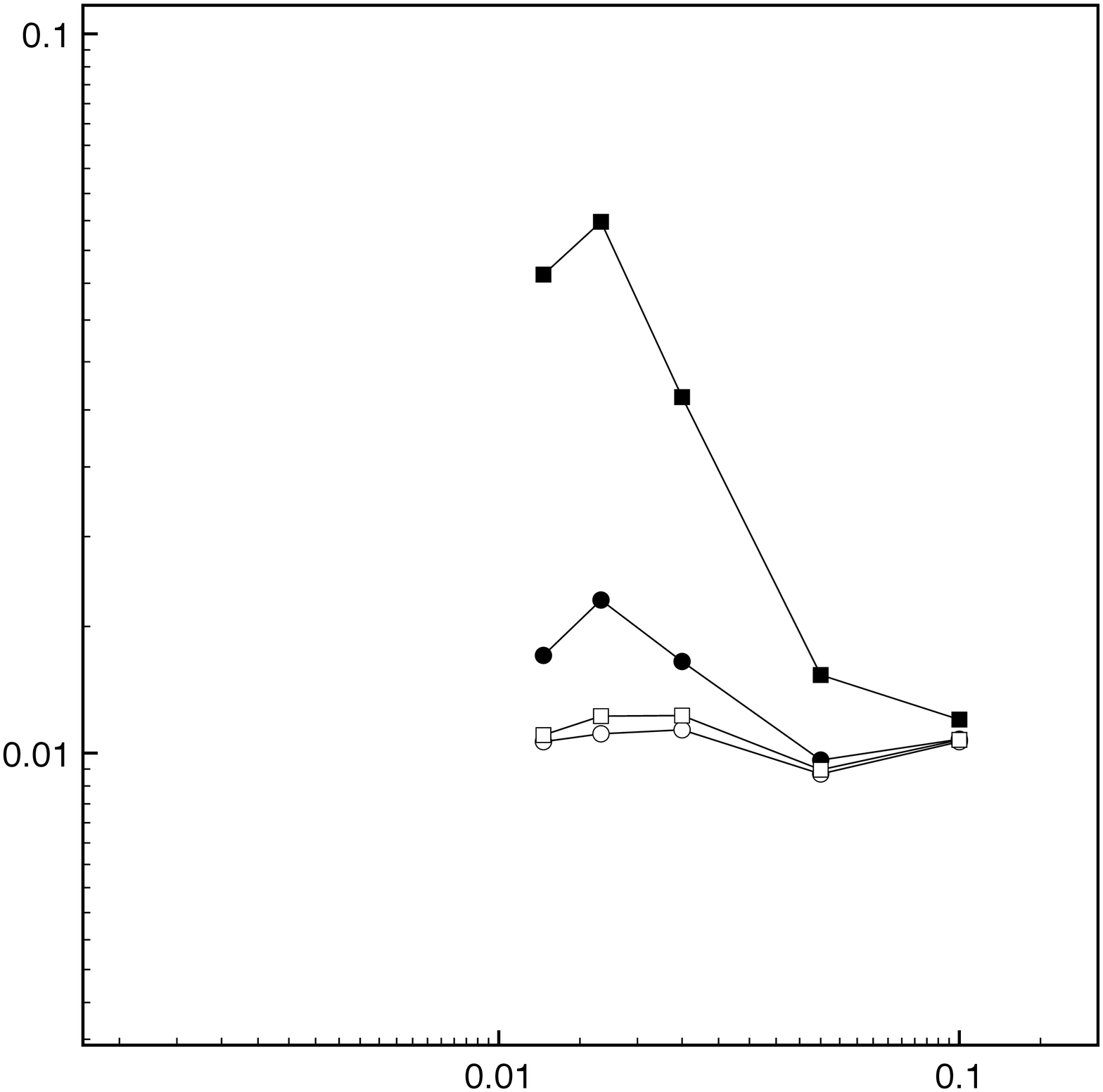}
\end{center}
\caption{Relative error against mesh size. Left graphic $k=1$, Neumann data with $2\%$
  relative perturbation. Right
  graphic $k=2$, Neumann data with $2\%$
  relative perturbation.  Square markers indicate 
  $H^1$-seminorm errors, circle markers indicate $L^2$-errors, filled
  markers indicate global errors and not filled markers indicate
  local errors. }\label{n3pert}
\end{figure}
\begin{figure}[hbt]
\begin{center}
\includegraphics[width=2.5in]{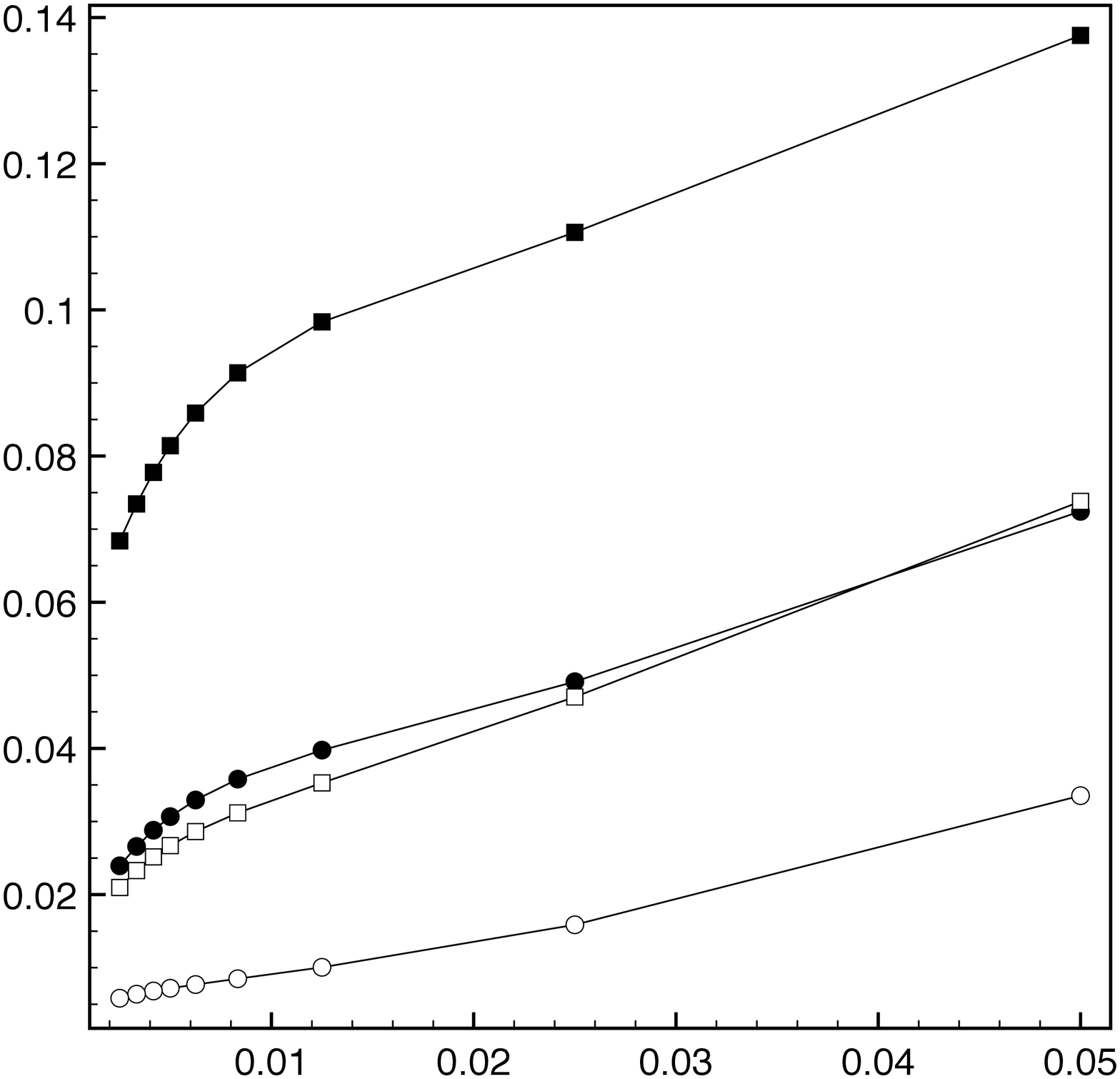}
\includegraphics[width=2.5in]{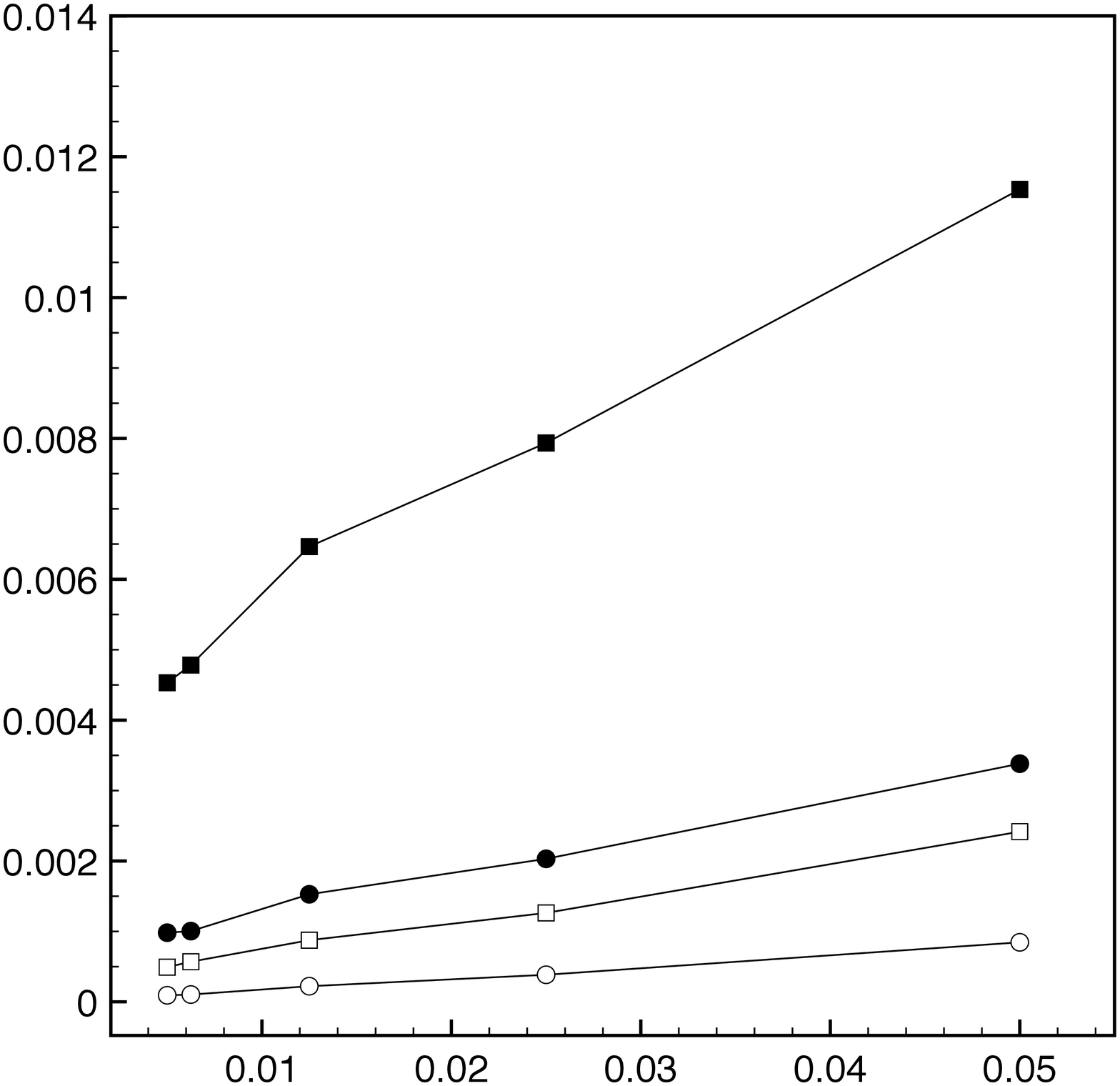}
\end{center}
\caption{Relative error against mesh size for case 2. From left to
  right: $k=1$ and $k=2$. Square markers indicate 
  $H^1$-seminorm errors, circle markers indicate $L^2$-errors, filled
  markers indicate global errors and not filled markers indicate
  local errors. }\label{case2}
\end{figure}
\section{Conclusion}
We have derived error estimates for a primal dual mixed finite element
method applied to the elliptic Cauchy problem. The results are optimal
with respect to the approximation orders of the finite element spaces
and the stability of the ill-posed problem. Introducing a special dual
stabilizer we reduce the scheme to a least squares mixed method for
which the number of degrees of freedom is significantly smaller, the
system matrix is symmetric, but the exact local flux conservation is
lost. This method satisfies similar estimates, but the results require slightly
more regularity of the source term and have slightly worse sensitivity to
perturbed data. We then showed that the reduced method can be used in
an iterative method to solve the full primal dual formulation, thus
recovering local conservation. The estimates show that if the exact solution is smooth
the use of high order approximation can pay off. However the
amplification of perturbations in data is also stronger with increased approximation
order. In numerical experiments we observed that the gain
obtained from the high order approximation is more important than
the increased sensitivity. Indeed both methods
performed better with higher order approximation, in particular for
oscillating solutions. As expected from the estimates the accuracy for
both methods was similar in our experiments, where the right hand side
was zero. The improved local conservation of the full primal dual
formulation was observed and the iterative procedure converged to a
relative residual of the increment in the $L^2$-norm of $O(10^{-6})$
within up to three iterations. Finally we point out the the method
presented herein also can be applied to inverse problems subject to
the Helmholtz equation, as those discussed in \cite{BNO17}.

\section*{Appendix}
Let $\lambda_{min}(A)$ and $\lambda_{max}(A)$ denote the smallest and
largest eigenvalues of the matrix $A$.
Assume, without loss of generality, that no $\tilde F$ has a corner of the domain through its interior.
For a patch $\tilde F$ let $\mathcal{N}_{\tilde F}$ denote the set of elements
with one face entirely in $\tilde F$, i.e. not touching the boundary
of $\tilde F$. Let $\mathcal{N}_{\tilde P}$ be the union of the elements
$\mathcal{N}_{\tilde F}$
and their interior neighbours, that is, any element
$K$ such that $K \cap \Sigma = \emptyset$ and $K \cap K' \ne \emptyset$ for some $K' \in
\mathcal{N}_{\tilde F}$. We also introduce the set
$\mathcal{N}_{\partial \tilde F}$ of elements in $\mathcal{N}_{\tilde
  F}$ with a neighbour that intersects $\partial \tilde F$. We define the patch $\tilde P:= \cup_{K \in
  \mathcal{N}_{\tilde P}}$. Now let $\tilde \varphi \in V_0^1$ such
that $\tilde \varphi\vert_{\partial \tilde P} = 0$ and $\tilde
\varphi(x_P) = 1$ for any interior vertex $x_P$ in $\tilde P$. It
follows that $\|\nabla \tilde \varphi\|_{\tilde P}^2 \lesssim h^{d-2}$. We will
first prove, using shape regularity and the properties of $A$, that provided $\mbox{diam}(\tilde F)/h$ is
large enough (but independently of $h$) there exists $c_0$, independent of $h$, such that
\begin{equation}\label{eq:low_bound}
c_0 h^{-1} \leq \mbox{meas}_{d-1}(\tilde F)^{-1} \int_{\tilde F} A\nabla \tilde \varphi \cdot \bfnu ~\mbox{d}
s =: \Theta(A,\tilde \varphi).
\end{equation}
Here we used that, for any element in $\mathcal{N}_{\tilde
  F}\setminus \mathcal{N}_{\partial \tilde F}$, $A\nabla \tilde \varphi \cdot \bfnu \ge
\lambda_{min}(A) |\nabla \tilde \varphi|$, with $h^{-1} \lesssim |\nabla \tilde
\varphi|$ on the face intersecting $\tilde F$.
\begin{multline*}
\int_{\tilde F} A\nabla \tilde \varphi \cdot \bfnu ~\mbox{d}
s = \sum_{K \in \mathcal{N}_{\partial \tilde F}} \int_{\partial K \cap
  \tilde F} A\nabla \tilde \varphi \cdot \bfnu ~\mbox{d}
s+\sum_{K \in \mathcal{N}_{\tilde F}\setminus \mathcal{N}_{\partial \tilde F} } \int_{\partial K \cap
  \tilde F} A\nabla \tilde \varphi \cdot \bfnu ~\mbox{d}
s \\
\ge -\sum_{K \in \mathcal{N}_{\partial \tilde F}} \lambda_{max}(A)
c_{max} h^{-1} h^{d-1} + \sum_{K \in\mathcal{N}_{\tilde F}\setminus \mathcal{N}_{\partial \tilde F} } \lambda_{min}(A)
c_{min} h^{-1} h^{d-1}
\end{multline*}
where $c_{min}$ and $c_{max}$ only depend on the shape regularity of
the elements. Observing that
$\mbox{card}(\mathcal{N}_{\partial \tilde F}) = O(h^{2-d})$ and
$\mbox{card}(\mathcal{N}_{\mathcal{N}_{\tilde F}\setminus
  \mathcal{N}_{\partial \tilde F}}) = O(h^{1-d})$ we see that the
second sum dominates the first for $\mbox{diam}(\tilde F)/h$
large enough. This concludes the proof of \eqref{eq:low_bound}.

Define $\varphi_{\tilde F} := \tilde \varphi/\Theta(A,\tilde
\varphi)$. By construction
\[
\int_{\tilde F} A \nabla \varphi_{\tilde F} \cdot \bfnu ~\mbox{d}s = \mbox{meas}_{d-1}(\tilde F).
\]
Consider now the $H^1$-seminorm of $\varphi_{\tilde
  F}$  on $\tilde P$,
\begin{equation}\label{eq:H1_semibound}
\|\nabla \varphi_{\tilde
  F}\|_{\tilde P} = \|\nabla \tilde \varphi\|_{\tilde P}/(c_0 h^{-1})^2\lesssim
(h^{d-2}/(c_0 h^{-1})^2)^{-\frac12}
\lesssim h^{\frac{d}{2}}.
\end{equation}
Using a Poincar\'e inequality we have $\|\varphi_{\tilde
  F}\|_{\tilde P} \lesssim h\|\nabla \varphi_{\tilde
  F}\|_{\tilde P}$ which together with \eqref{eq:H1_semibound} yields
the desired bound \eqref{eq:vphistab}.

\section*{Acknowledgment} 
EB was supported in part by the EPSRC grant EP/P01576X/1.
LO was supported by EPSRC grants EP/L026473/1 and EP/P01593X/1. 
ML was supported in part by the Swedish Foundation for Strategic 
Research Grant No.\ AM13-0029, the Swedish Research Council 
Grants Nos.\  2013-04708, 2017-03911, and the Swedish Research 
Programme Essence.
\bibliographystyle{plain}
\bibliography{LMS}

\end{document}